\documentclass[12pt,reqno]{amsart}

\newcommand\version{May 1, 2023}

\usepackage{amsmath,amsfonts,amsthm,amssymb,amsxtra}
\usepackage{hyperref}
\usepackage{bbm} 
\usepackage{bm} 
\usepackage{color}
\usepackage{accents} 
\makeindex 

\newtheorem{theorem}{Theorem}[section]
\newtheorem{proposition}[theorem]{Proposition}
\newtheorem{lemma}[theorem]{Lemma}
\newtheorem{corollary}[theorem]{Corollary}

\theoremstyle{definition}

\newtheorem{definition}[theorem]{Definition}

\newtheorem{assumption}[theorem]{Assumption}

\theoremstyle{remark}

\newtheorem{remark}[theorem]{Remark}
\newtheorem{remarks}[theorem]{Remarks}

\numberwithin{equation}{section}

\newcommand{\C}{\mathbb{C}}

\renewcommand{\epsilon}{\varepsilon}

\newcommand{\F}{\mathcal{F}}

\newcommand{\loc}{{\rm loc}}

\newcommand{\N}{\mathbb{N}}

\renewcommand{\phi}{\varphi}
\newcommand{\R}{\mathbb{R}}

\newcommand{\Z}{\mathbb{Z}}

\newcommand{\floor}[1]{\left \lfloor{#1}\right \rfloor }

\DeclareMathOperator{\re}{Re}
\DeclareMathOperator{\spann}{span}

\DeclareMathOperator{\supp}{supp}
\DeclareMathOperator{\sgn}{sgn}

\def\bs{\mathbb{S}}

\def\ca{\mathcal{A}}
\def\cb{\mathcal{B}}
\def\ce{\mathcal{E}}

\def\ci{\mathcal{I}}
\def\ck{\mathcal{K}}

\def\cm{\mathcal{M}}

\def\Rd{{\mathbb{R}^d}}

\def\ga{\mathfrak{A}}
\def\gb{\mathfrak{B}}
\def\gE{\mathfrak{E}}

\newcommand{\me}[1]{\mathrm{e}^{#1}}
\newcommand{\one}{\mathbf{1}}

\newcommand\klm[1]{\textcolor{blue}{#1}}

\makeatletter
\newcommand*{\rom}[1]{\expandafter\@slowromancap\romannumeral #1@}
\makeatother

\begin{document}

\title[Ground state representation --- \version]{Ground state representation for the fractional Laplacian with Hardy potential in angular momentum channels}

\author[K. Bogdan]{Krzysztof Bogdan}
\address[Krzysztof Bogdan]{Department of Pure and Applied Mathematics, Wroc\l aw University of Science and Technology, Hoene-Wro\'nskiego 13C, 50-376 Wroc\l aw, Poland}
\email{krzysztof.bogdan@pwr.edu.pl}

\author[K. Merz]{Konstantin Merz}
\address[Konstantin Merz]{Institut f\"ur Analysis und Algebra, Technische Universit\"at Braunschweig, Universit\"atsplatz 2, 38106 Braun\-schweig, Germany, and Department of Mathematics, Graduate School of Science, Osaka University, Toyonaka, Osaka 560-0043, Japan}
\email{k.merz@tu-bs.de}



\date{\version}

\begin{abstract}
  Motivated by the study of relativistic atoms, we consider the Hardy operator $(-\Delta)^{\alpha/2}-\kappa|x|^{-\alpha}$ acting on functions of the form $u(|x|) |x|^{\ell} Y_{\ell,m}(x/|x|)$ in $L^2(\R^d)$,
  when $\kappa\geq0$ and $\alpha\in(0,2]\cap(0,d+2\ell)$. We give a ground state representation of the corresponding form on the half-line (Theorem~\ref{gstransformreformulated}). For the proof we use subordinated Bessel heat kernels.
\end{abstract}

\thanks{The first-named author was supported by grant 2017/27/B/ST1/01339 of National Science Centre, Poland.
  The second-named author was supported through the PRIME programme of the German Academic Exchange Service (DAAD) with funds from the German Federal Ministry of Education and Research (BMBF).
  We thank Volker Bach and Haruya Mizutani for helpful discussions.}

\maketitle
\setcounter{tocdepth}{1}
\tableofcontents

\section{Introduction and main result}

\subsection{Classic Hardy--Kato--Herbst inequalities and ground state representations}

We let $d\in\N:=\{1,2,\ldots\}$. For $\alpha\in(0,2)$, we consider the quadratic form
\begin{align}
  \label{eq:fractlaplacesingularintegral}
  \ce^{(\alpha)}[f]
  :=\frac12\int_{\R^d}\!\int_{\R^d} 
  |f(x)-f(y)|^2\nu(x,y)
  \,dy\,dx, \quad f\in L^2(\R^d),
\end{align}\index{$\ce^{(\alpha)}$}where
\begin{equation*}
  \nu(x,y) := \nu(x-y),\quad x,y\in \R^d,
\end{equation*}\index{$\nu(x,y)$}with
\begin{align}
  \label{eq:defaalphad}
  \nu(z):=\mathcal{A}_{d,-\alpha}|z|^{-d-\alpha}
  \quad 
  \text{and}
  \quad
  \mathcal{A}_{d,-\alpha}
  := \frac{2^{\alpha}\Gamma\big((d+\alpha)/2\big)}{\pi^{d/2}|\Gamma(-\alpha/2)|}.
\end{align}\index{$\ca_{d,-\alpha}$}There is an equivalent representation of $\ce^{(\alpha)}$ using operator semigroups. With
\begin{align}
  \label{eq:defusualheatkernel}
  P^{(\alpha)}(t,x,y) := \frac{1}{(2\pi)^d}\int_{\R^d}d\xi\, \me{-i\xi\cdot(x-y)-t|\xi|^\alpha} , \quad x,y\in\R^d,\,t>0,
\end{align}\index{$P^{(\alpha)}(t,x,y)$}we have
\begin{align}
  \label{eq:fractlaplaceheatkernel}
  \ce^{(\alpha)}[f]
  = \lim_{t\to0}\frac1t\langle f,(1-P^{(\alpha)}(t,\cdot,\cdot))f\rangle_{L^2(\R^d)}, \quad f\in L^2(\R^d).
\end{align}
Here $\langle f,g\rangle_{L^2(\R^d)}=\int_{\R^d}\overline{f(x)}g(x)\,dx$ and, as usual,
\begin{align*}
  (P^{(\alpha)}(t,\cdot,\cdot)f)(x) := \int_{\R^d} P^{(\alpha)}(t,x,y)f(y)\,dy, \quad t>0,\,x\in\R^d.
\end{align*}
Moreover, using the Fourier transform
\begin{align*}
  \hat f(\xi):=(\F f)(\xi):=(2\pi)^{-d/2}\int_{\R^d}\me{-ix\cdot\xi}f(x)\,dx,
\end{align*}\index{$\hat f$}\index{$\F$}initially defined on $L^1(\R^d)$, and extended to a unitary operator on $L^2(\R^d)$, we get
\begin{align}
  \label{eq:dirichletformalpha}
  \begin{split}
    \ce^{(\alpha)}[f]
    & = \int_{\R^d} |\hat f(\xi)|^2|\xi|^\alpha\,d\xi
      \quad \text{for all}\ f\in H^{\alpha/2}(\R^d),
  \end{split}
\end{align}
with the usual definition of Sobolev spaces
\begin{align}
  H^s(\R^d) := \left\{f\in L^2(\R^d):\,
  \int_{\R^d} |\xi|^{2s} |\hat f(\xi)|^2\,d\xi < \infty\right\}, \quad s\geq0.
\end{align}\index{$H^s(\R^d)$}We also consider the corresponding homogeneous Sobolev spaces $\dot H^s(\R^d)$\index{$\dot H^{s}(\R^d)$} for $s\geq0$, defined as the completion of $C_c^\infty(\R^d)$ under
$\||\xi|^s \hat f\|_{L^2(\R^d)}$.
When $\alpha=2$, we write
\begin{align}
  \label{eq:deflaplacequadform}
  \ce^{(2)}[f]:=\langle\nabla f,\nabla f\rangle_{L^2(\R^d)}, \quad f\in \dot H^1(\R^d).
\end{align}\index{$\ce^{(2)}$}Since each component of $D:=-i\nabla$ in $L^2(\R^d)$ is self-adjoint with domain $H^1(\R^d)$ and they commute, the spectral theorem ensures that $|D|^\alpha=(-\Delta)^{\alpha/2}$ with domain $H^\alpha(\R^d)$ is also self-adjoint. Moreover, $|D|^\alpha$ coincides with the Friedrichs extension corresponding to the quadratic form $\ce^{(\alpha)}$, initially defined on $C_c^\infty(\R^d)$. In particular, the kernel in \eqref{eq:defusualheatkernel} is just the integral kernel of $\me{-t|D|^\alpha}$.

\smallskip
Here and below we always assume that $\alpha\in(0,2]$. If $\alpha<d$, which is an additional restriction only for $d=1$ and $2$, then the classical Hardy--Kato--Herbst inequality in $\R^d$---in short, Hardy ine\-qua\-lity---asserts that the lower bound
\begin{align}
  \label{eq:hardy}
  \ce^{(\alpha)}[f]
  \geq \kappa \int_{\R^d} \frac{|f(x)|^2}{|x|^\alpha}\,dx
\end{align}
holds for all $f\in C_c^\infty(\R^d)$ if and only if
the coupling constant $\kappa$ satisfies
$\kappa\leq\kappa_{\rm c}^{(\alpha)}$ with
the critical coupling constant
\begin{align}
  \label{eq:defgammacellequalzero}
  \kappa_{\rm c}^{(\alpha)}
  := \frac{2^{\alpha} \Gamma \left((d+\alpha)/4\right)^2}{\Gamma \left((d-\alpha)/4\right)^2}.
\end{align}\index{$\kappa_{\rm c}^{(\alpha)}$}In particular, $\kappa_{\rm c}^{(2)}=(d-2)^2/4$.
Since $C_c^\infty(\R^d)$ is dense in $\dot H^{s}(\R^d)$ when $s\in(0,d/2)$ (see, e.g., \cite[Theorem~1.70]{Bahourietal2011}),
Hardy's inequality \eqref{eq:hardy} holds for all $\kappa\leq\kappa_{\rm c}^{(\alpha)}$ and $f\in \dot H^{\alpha/2}(\R^d)$. In fact, \eqref{eq:hardy} holds for $f\in L^2(\R^d)$ if $0<\alpha<\min\{2,d\}$. For references, see, e.g., \cite{Hardy1919,Hardy1920,Kato1966,Herbst1977,Kovalenkoetal1981,Yafaev1999,Franketal2008H,FrankSeiringer2008,Bogdanetal2016}.

One of the many proofs of \eqref{eq:hardy} uses what is customarily called the \emph{ground state representation}, that is, the decomposition
\begin{subequations}
  \label{eq:orggstransformtogether}
  \begin{align}
    \label{eq:orggstransform}
    \begin{split}
      \ce^{(\alpha)}[|x|^{-\sigma} \cdot f]
      & = \frac{2^{\alpha} \Gamma \left(\frac{d-\sigma}{2}\right) \Gamma\left(\frac{\alpha+\sigma}{2}\right)}{\Gamma \left(\frac{\sigma}{2}\right) \Gamma \left(\frac{d-\alpha-\sigma}{2}\right)} \int_{\R^d}\frac{||x|^{-\sigma} \cdot f(x)|^2}{|x|^\alpha}\,dx \\
      & \quad + \frac{1}{2}\iint\limits_{\R^d\times \R^d}\left|f(x)-f(y)\right|^2 \nu(x-y) (|x||y|)^{-\sigma}\,dx\,dy
    \end{split}
  \end{align}
  if $\alpha<\min\{2,d\}$ and
  \begin{align}
    \label{e.gs2}
    \ce^{(2)}[|x|^{-\sigma} \cdot f] = \sigma(d-\sigma-2)\int_{\R^d}\frac{||x|^{-\sigma} \cdot f(x)|^2}{|x|^2}\,dx + \int_{\R^d}\left|\nabla f(x)\right|^2\cdot |x|^{-2\sigma}\,dx
  \end{align}
\end{subequations}
if $\alpha=2<d$.
Here $\sigma\in[0,d-\alpha]$ and $f\in C_c^\infty(\R^d\setminus\{0\})$. In fact, for $\alpha<d$, by the density of $C_c^\infty(\R^d\setminus\{0\})$ in $\dot H^{\alpha/2}(\R^d)$ (see, e.g., \cite[Theorem~1.70]{Bahourietal2011}), Formulae~\eqref{eq:orggstransform}--\eqref{e.gs2} hold for all $f\in\dot H^{\alpha/2}(\R^d)$; see, e.g., \cite[Theorem~1.2]{FrankSeiringer2008} or \cite[Lemma~4.4]{Dipierroetal2016}.
For $\alpha=2$, the representation \eqref{e.gs2} has been known long since, see, e.g., \cite[p.~169]{ReedSimon1975} for a textbook treatment of $d=3$. For $d=3$ and $\alpha=1$, it was derived, e.g., by Michael Loss (unpublished notes, cf.~\cite[p.~935]{Franketal2008H}), while the general case was proved by Frank, Lieb, and Seiringer \cite [Proposition 4.1]{Franketal2008H} and, by different methods, by Bogdan, Dyda, and Kim \cite[Proposition~5]{Bogdanetal2016}, who also showed that \eqref{eq:orggstransform} holds for all $f\in L^2(\R^d)$. 

Note that the prefactors of the first term on the right-hand sides of \eqref{eq:orggstransform}--\eqref{e.gs2}, regarded as functions of $\sigma\in[0,d-\alpha]$, are symmetric about $\sigma=(d-\alpha)/2$, strictly increasing for $\sigma\in[0,(d-\alpha)/2]$ with maximal value $\kappa_c^{(\alpha)}$ at $\sigma=(d-\alpha)/2$, and vanish at $\sigma=0$.
Since the second term on the right sides of \eqref{eq:orggstransform}--\eqref{e.gs2} vanishes if $f=|x|^{-\sigma}$, we call the radial function $|x|^{-\sigma}$ a \emph{generalized ground state} for $|D|^\alpha-\Phi(\sigma)|x|^{-\alpha}$, although this abuses standard terminology because $|x|^{-\sigma}\notin L^2(\R^d)$.

As we mentioned, \eqref{eq:orggstransform}--\eqref{e.gs2} provide proofs of the Hardy inequality \eqref{eq:hardy} for $f\in C_c^\infty(\R^d\setminus\{0\})$. They also imply that the constant $\kappa=\kappa_{\rm c}$ in \eqref{eq:hardy} is optimal; see \cite[Remark~4.2]{Franketal2008H} or \cite{Bogdanetal2022}. Furthermore, \eqref{eq:orggstransform}--\eqref{e.gs2} reveal that equality in \eqref{eq:hardy} cannot hold  for nonzero functions for which both sides of \eqref{eq:hardy} are finite, see, e.g., \cite[Section~3]{Yafaev1999}.

\begin{remark}
  The ground state representations \eqref{eq:orggstransform}--\eqref{e.gs2} express the quadratic form of $|D|^\alpha-\Phi(\sigma)/|x|^\alpha$ in $L^2(\R^d)$ as the quadratic form of $|D|^\alpha$ in a weighted $L^2$ space with weight $|x|^{-\sigma}$. The astute reader may therefore discern in \eqref{eq:orggstransform}--\eqref{e.gs2} Doob's conditioning and Schr\"{o}dinger perturbations of $|D|^\alpha$; see \cite{Fitzsimmons2000,Bogdanetal2016}.
\end{remark}

\subsection{Angular momentum decomposition}
\label{ss:mainresult}
Our goal is to analyze forms
\begin{align}
  \label{eq:defhardyop}
  \ce^{(\alpha)}[f] - \kappa \int_{\R^d}\frac{|f(x)|^2}{|x|^\alpha}\,dx,
\end{align}
by taking into account the spherical symmetry of the operators $|D|^\alpha$ and $|x|^{-\alpha}$, i.e., the fact that they commute with the generator of rotations---sometimes called angular momentum operator---and therefore with its square, the Laplace--Beltrami operator on $L^2(\bs^{d-1})$. The latter is a non-negative operator with purely discrete spectrum and orthonormal eigenbasis consisting of spherical harmonics. The spherical symmetry allows to decompose $|D|^\alpha$ and $|x|^{-\alpha}$ into a direct sum of operators, each acting on radial functions multiplied by a spherical harmonic.

The main result of this paper is a ground state representation, in Theorem~\ref{gstransformreformulated} below, of the quadratic form \eqref{eq:defhardyop} for functions $f$ in a fixed angular momentum channel. To state the theorem, let us make our notions precise.
For $d\in\N$, we consider the orthonormal basis of (hyper)spherical harmonics in $L^2(\bs^{d-1})$ on the $(d-1)$-dimensional unit sphere $\bs^{d-1}$. For $d=1$, $\bs^0=\{-1,+1\}$, $L^2(\bs^0)$ is equipped with the counting measure, and the spherical harmonics are simply $Y_{0}(\omega)=2^{-1/2}$ and $Y_{1}(\omega)=2^{-1/2}\omega$, $\omega\in\{-1,+1\}$.
When $d=2$, they consist of $\me{\pm i \ell \omega}/\sqrt{2\pi}$ with $\ell\in\N_0$ and $\omega$ in the interval $[0,2\pi)$, which we identify with $\bs^1$. If $d\geq 3$ then the spherical harmonics are $\{Y_{\ell,m}\}_{\ell\in L_d,m\in M_{d,\ell}}$, with $L_d:=\N_0$ and $M_{3,\ell}:=\{-\ell,-\ell+1,...,\ell-1,\ell\}$, $M_{d,\ell} := \{(m_1,...,m_{d-2})\in\Z^{d-2}:\,|m_{d-2}|\leq \cdots\leq m_2\leq m_1\leq\ell\}$ for $d\geq4$; see, e.g., \cite[p.~34]{AveryAvery2018}.\index{$L_d$}\index{$M_\ell$} We extend the notation to $d=2$ by letting $L_2=\N_0$, $M_{2,0}=\{0\}$, $M_{2,\ell}=\{-1,+1\}$ for $\ell\in \N$, and $Y_{\ell,m}(\omega):=\me{i m \ell \omega}/\sqrt{2\pi}$, where $\omega\in[0,2\pi)$, $\ell\in L_2$, and $m\in M_{2,\ell}$. Finally, we let $L_1=\{0,1\}$ and $M_{1,0}=M_{1,1}=\{0\}$, just to enumerate $Y_0=:Y_{0,0}$ and $Y_1=:Y_{1,0}$ for $d=1$.
In the following, we write $M_\ell$ instead of $M_{d,\ell}$, if there is no danger of confusion.

For $\R_+=(0,\infty)$, $\ell\in L_d$, $m\in M_\ell$, and almost everywhere ($a.e.$) defined Borel function $u:\R_+\to \C$, we let
\begin{align}\index{$[u]_{\ell,m}$}
  \label{eq:defeqclassellm}
  [u]_{\ell,m}(x) := u(|x|) |x|^{\ell} Y_{\ell,m}(\omega_x),
  \quad \mbox{ for } a.e.\;\;  x\in\R^d\setminus\{0\}, 
\end{align}
\noindent
where $\omega_x=x/|x|$ for $x\in\R^d\setminus\{0\}$.\index{$\omega_x$} In what follows we usually skip the adverbs $a.e.$, according to the fact that we are mainly interested in equivalence classes of functions equal $a.e.$, in particular in the square-integrable ``functions'' on $\Rd$, $u\in L^2(\R^d)$. Here, as usual, if the reference measure is not detailed, we mean the Lebesgue measure.
Note that $[1]_{\ell,m}(x)$ are the (regular) solid harmonics $|x|^\ell Y_{\ell,m}(\omega_x)$---they are indeed harmonic, $\Delta[1]_{\ell,m}=0$.
Using polar coordinates, one observes that
\begin{align}
  \label{eq:isometry}
  \|[u]_{\ell,m}\|_{L^2(\R^d)} = \|u\|_{L^2(\R_+,r^{d-1+2\ell}dr)},
\end{align}
where $r^{d-1+2\ell}dr$ indicates the reference measure for the $L^2$ space on $\R_+$.
Let
\begin{equation}
  \label{eq:defvelll2}
  V_{\ell,m}:= \{ [u]_{\ell,m}:\, u\in L^2(\R_+,r^{d-1+2\ell}dr)\},
  \quad \ell\in L_d, \; m\in M_\ell,
\end{equation}\index{$V_{\ell,m}$}and 
\begin{equation}
  \label{eq:defvelll2t}
  V_{\ell}:= \spann\{ [u]_{\ell,m}:\, u\in L^2(\R_+,r^{d-1+2\ell}dr), \; m\in M_\ell\},
  \quad \ell\in L_d.
\end{equation}\index{$V_\ell$}Clearly, $V_{\ell,m}\subset V_{\ell}\subset L^2(\R^d)$.
We call $V_\ell$ \emph{angular momentum channel}, and say that functions in $V_{\ell}$ have \emph{angular momentum $\ell$}. For instance, for $d=1$, $V_0$ are the even and $V_1$ are the odd functions in $L^2(\R^1)$.
We will not introduce terminology for $V_{\ell,m}$, since $m\in M_\ell$ will play a rather minor role:
The values of $\ce^{(\alpha)}[[u]_{\ell,m}]$ and $\int_{\R^d}|[u]_{\ell,m}(x)|^2|x|^{-\alpha}\,dx$ are independent of $m\in M_{\ell}$, see Remarks~\ref{remarksmainresult} below.

Since the spherical harmonics form an orthonormal basis of $L^2(\bs^{d-1})$, we get
\begin{align}
  \label{eq:directsumvell}
  L^2(\R^d) = \bigoplus_{\ell\in L_d}V_{\ell}
  \qquad \text{and} \qquad
  V_\ell= \bigoplus_{m\in M_\ell}  V_{\ell,m}.
\end{align}
Accordingly, for every $f\in L^2(\R^d)$, there are unique \emph{coordinate functions} or \emph{coefficients} $f_{\ell,m}\in V_{\ell,m}$, $\ell\in L_d$, $m\in M_\ell$, such that
\begin{align}
  \label{eq:l2expansionsphericalharmonics}
  f = \sum_{\ell\in L_d}\sum_{m\in M_\ell}f_{\ell,m}, \qquad
  \|f\|_{L^2(\R^d)}^2 = \sum_{\ell\in L_d}\sum_{m\in M_\ell}\|f_{\ell,m}\|_{L^2(\R^d)}^2.
\end{align}
By \eqref{eq:defeqclassellm} and \eqref{eq:isometry}, we have
\begin{align}
  \label{eq:expressioncoordinatefunction}
  f_{\ell,m} = [u_{\ell,m}]_{\ell,m},
\end{align}
where
\begin{align}
  u_{\ell,m}(r) = r^{-\ell} \int_{\bs^{d-1}}d\omega\, \overline{Y_{\ell,m}(\omega)}\, f_{\ell,m}(r\omega).
\end{align}
Moreover, $\|f_{\ell,m}\|_{L^2(\R^d)}=\|u_{\ell,m}\|_{L^2(\R_+,r^{d-1+2\ell}dr)}<\infty$. The reader may use \eqref{eq:expressioncoordinatefunction} to reformulate \eqref{eq:l2expansionsphericalharmonics}.
The following lemma provides an \emph{angular momentum decomposition} of $\ce^{(\alpha)}[f]$ and quadratic forms corresponding to spherically symmetric multiplication operators, such as $|x|^{-\alpha}$. It says that these forms break down into forms corresponding to different spaces $V_{\ell,m}$.

\begin{lemma}
  \label{fourierbesselcor}
  Let $d\!\in\!\N$, $\alpha\!\in\!(0,2]$, and $f\!\in\! H^{\frac\alpha2}(\R^d)$ with the expansion \eqref{eq:l2expansionsphericalharmonics}. Then,
  \begin{subequations}
    \begin{align}
      & \ce^{(\alpha)}[f] = \sum_{\ell\in L_d}\sum_{m\in M_\ell}\ce^{(\alpha)}[f_{\ell,m}]
        \quad \text{and} \quad \\
      & \int_{\R^d}\frac{|f(x)|^2}{|x|^\alpha}\,dx
        = \sum_{\ell\in L_d}\sum_{m\in M_\ell}\int_{\R^d}\frac{|f_{\ell,m}(x)|^2}{|x|^\alpha}\,dx.
    \end{align}
  \end{subequations}
\end{lemma}
The statement easily follows from the orthonormality of $\{Y_{\ell,m}\}_{\ell\in L_d,m\in M_\ell}$ in $L^2(\bs^{d-1})$, see also Lemma~\ref{fourierbessel} below.

\smallskip
It is well known that there is no (non-trivial) Hardy inequality for the Laplacian when $d=1$ or $d=2$. However, the presence of angular momentum allows to prove non-trivial Hardy inequalities for $\alpha<d+2\ell$.
Accordingly, we will treat \\
(i) $\alpha\in(0,1)$ when $d=1$ and $\ell=0$, \\
(ii) $\alpha\in(0,2)$ when $d=2$ and $\ell=0$, \\
(iii) $\alpha\in(0,2]$ when $d=\ell=1$, or $d=2$ and $\ell\geq1$, or $d\geq3$ and $\ell\geq0$.

The following definition proposes a convenient parameterization of the coupling constant $\kappa$ in \eqref{eq:defhardyop}, suitable for dealing with \emph{generalized ground states in angular momentum channels $V_\ell$}.
\begin{definition}
  \label{defsigmagammaell}
  Let $d\in\N$ and $\ell\in L_d$. For $\alpha\in(0,\min\{2,d+2\ell\})$ we define
  \begin{align}
    \label{eq:defsigmagammal}
    \Phi_{d,\ell}^{(\alpha)}(\sigma) 
    := \frac{2^{\alpha} \Gamma\left(\frac{\ell+\sigma+\alpha}{2}\right) \Gamma \left(\frac{d+\ell-\sigma}{2}\right)}{\Gamma \left(\frac{\ell+\sigma}{2}\right) \Gamma \left(\frac{d+\ell-\sigma-\alpha}{2}\right)},
    \quad \sigma \in (-\ell-\alpha,d+\ell),
  \end{align}\index{$\Phi_{d,\ell}^{(\alpha)}(\sigma)$}and, considering the midpoint $\sigma=(d-\alpha)/2$, we let
  \begin{align}\index{$\kappa_{\rm c}^{(\alpha)}(\ell)$}
    \label{eq:defgammacell}
    \kappa_{\rm c}^{(\alpha)}(\ell) := \Phi_{d,\ell}^{(\alpha)}\left(\frac{d-\alpha}{2}\right)=\frac{2^{\alpha} \Gamma \left(\frac{1}{4} (d+2\ell+\alpha)\right)^2}{\Gamma \left(\frac{1}{4} (d+2\ell-\alpha)\right)^2}.
  \end{align}
  We also define, for $\alpha=2<d+2\ell$, 
  \begin{align}
    \Phi_{d,\ell}^{(2)}(\sigma):=(\ell+\sigma)(d+\ell-\sigma-2),\quad  \sigma\in \R,
  \end{align}
  and we let $\kappa_{\rm c}^{(2)}(\ell):=(d-2+2\ell)^2/4$.
\end{definition}

Here are some properties of the map $\sigma\mapsto\Phi_{d,\ell}^{(\alpha)}(\sigma)$, taking into account qualitative differences between the cases $\alpha=2$ and $\alpha<2$.

\begin{proposition}
  \label{propertiesphi}
  Let $d\in\N$, $\ell\in L_d$, $\alpha\in(0,2]\cap(0,d+2\ell)$, $S=\alpha$ if $\alpha<2$, and $S=\infty$ if $\alpha=2$. Let $\sigma\in(-\ell-S,d-\alpha+\ell+S)$.
  Then $\Phi_{d,\ell}^{(\alpha)}(\sigma)=\Phi_{d+2\ell,0}^{(\alpha)}(\sigma+\ell)$.
  The function $\sigma\mapsto\Phi_{d,\ell}^{(\alpha)}(\sigma)$ is symmetric about $\sigma=(d-\alpha)/2$,
  strictly increasing for $\sigma\leq(d-\alpha)/2$, and satisfies $\lim_{\sigma\searrow-\ell-S}\Phi_{d,\ell}^{(\alpha)}(\sigma)=-\infty$, $\Phi_{d,\ell}^{(\alpha)}(-\ell)=0$, $\Phi_{d,\ell}^{(\alpha)}((d-\alpha)/2)=\kappa_{\rm c}^{(\alpha)}(\ell)$.
  Moreover, for $\ell,\ell'\in L_d$ with $\ell>\ell'$ and $\sigma\in(-\ell'-S,(d-\alpha)/2]$,
  \begin{align}
    \label{eq:sigmagammaellneg}
    \Phi_{d,\ell}^{(\alpha)}(\sigma) > \Phi_{d,\ell'}^{(\alpha)}(\sigma).
  \end{align}
\end{proposition}
Note that $\kappa_{\rm c}^{(\alpha)}=\kappa_{\rm c}^{(\alpha)}(0)$ and $\Phi_{d,0}^{(\alpha)}(\sigma)$ is the prefactor in the first term on the right-hand sides of \eqref{eq:orggstransform}--\eqref{e.gs2}.

\begin{proof}
  Since $\Phi_{d,\ell}^{(\alpha)}(\sigma)=\Phi_{d+2\ell,0}^{(\alpha)}(\sigma+\ell)$ and the right-hand side is well-studied, e.g., in \cite[Lemma~3.2]{Franketal2008H}, \cite[p.~237]{Bogdanetal2016} and \cite[p.~1004---1005]{JakubowskiWang2020}, we only prove the monotonicity statement in \eqref{eq:sigmagammaellneg} when $\alpha<2$. To this end, we use $\Phi_{d,\ell}^{(\alpha)}(\sigma)=\Phi_{d+2\ell,0}^{(\alpha)}(\sigma+\ell)$ and the monotonicity of the functions $\sigma\mapsto\Phi_{d,0}^{(\alpha)}(\sigma)$ and $d\mapsto\Phi_{d,0}^{(\alpha)}(\sigma)$ for all $-\alpha<\sigma\leq(d-\alpha)/2$. As mentioned above, the former monotonicity is well known. To prove the latter, it suffices to show that, for arbitrary $a>0$, the function $(a,\infty)\ni x\mapsto \frac{\Gamma(x)}{\Gamma(x-a)}$ is increasing. But this follows from
  \begin{align*}
    \frac{d}{dx}\left(\log \frac{\Gamma(x)}{\Gamma(x-a)}\right)
    = \frac{\Gamma'(x)}{\Gamma(x)}-\frac{\Gamma'(x-a)}{\Gamma(x-a)}
    = \sum_{k\geq0}\left(\frac{1}{x-a+k}-\frac{1}{x+k}\right) > 0,
  \end{align*}
  where we used \cite[(5.7.6)]{NIST:DLMF}. This concludes the proof of Proposition~\ref{propertiesphi}.
\end{proof}

Since we regard $d\in\N$ and $\alpha\in(0,2]\cap(0,d+2\ell)$ as given, we shall write $\ce$\index{$\ce$}, $P(t,x,y)$\index{$P(t,x,y)$}, $\Phi_\ell(\sigma)$\index{$\Phi_\ell(\sigma)$}, $\kappa_{\rm c}(\ell)$\index{$\kappa_{\rm c}(\ell)$}, and $\kappa_{\rm c}$\index{$\kappa_{\rm c}$} instead of $\ce^{(\alpha)}$, $P^{(\alpha)}(t,x,y)$, $\Phi_{d,\ell}^{(\alpha)}(\sigma)$, $\kappa_{\rm c}^{(\alpha)}(\ell)$, and $\kappa_{\rm c}^{(\alpha)}$, respectively, if there is no danger of confusion. When $\ell=0$, we further abbreviate $\Phi(\sigma):=\Phi_0(\sigma)$\index{$\Phi(\sigma)$}.

\subsection{Main result}
Proposition~\ref{propertiesphi} implies, in particular, that for each $\kappa\leq\kappa_{\rm c}(\ell)$ there is a unique number $\sigma \leq (d-\alpha)/2$ such that $\kappa=\Phi_\ell(\sigma)$. Therefore, from now on, instead of \eqref{eq:defhardyop}, we mostly consider
\begin{align}
  \label{eq:hardyforml}
  \ce[f] - \Phi_\ell(\sigma) \int_{\R^d}\frac{|f(x)|^2}{|x|^\alpha}\,dx, \quad f\in C_c^\infty(\R^d\setminus\{0\}),
\end{align}
where $\ell\in L_d$, $\alpha\in(0,2]\cap(0,d+2\ell)$, and $\sigma\leq(d-\alpha)/2$. In fact, although the case of $\sigma<-\ell$ (corresponding to $\Phi_\ell(\sigma)<0$) is also of much interest (see, e.g., \cite{JakubowskiWang2020,Choetal2020}), we will exclusively deal with the \emph{attractive}, that is, positive (or rather \emph{non-repulsive}, that is, non-negative) coupling constants; more precisely, from now on we assume $\sigma\in[-\ell,(d-\alpha)/2]$, i.e., $\Phi_\ell(\sigma)\in[0,\kappa_c(\ell)]$.

\smallskip
Our main result is the following ground state representation or Hardy identity for the quadratic form \eqref{eq:hardyforml} in $V_{\ell,m}$.

\begin{theorem}
  \label{gstransformreformulated}
  Let $d\in\N$, $\ell\in L_d$, $m\in M_\ell$, $\alpha\in(0,2]\cap(0,d+2\ell)$, $\sigma\in [-\ell,(d-\alpha)/2]$, $h(r):=r^{-\ell-\sigma}$, and $u\in C_c^\infty(\R_+)$.
  For $\alpha<2$, we let $\nu_{\frac{d-1}{2}+\ell}(r,s)>0$ be as defined in Proposition~\ref{potentialexpansionsphericalharmonics} below. Then
  \begin{align}
    \label{eq:nuellmainthm}
    \nu_{\frac{d-1}{2}+\ell}(r,s)=(rs)^{-\ell}\iint_{\bs^{d-1}\times\bs^{d-1}}\overline{Y_{\ell,m}(\omega_x)}Y_{\ell,m}(\omega_y)\nu(r\omega_x,s\omega_y)\,d\omega_x\,d\omega_y
  \end{align}\index{$\nu_{\frac{d-1}{2}+\ell}(r,s)$}and we let
  \begin{subequations}
    \label{eq:defiellsigma}
    \begin{align}
      \label{eq:defiellsigmaalphal2}
      \ci_{\ell,\sigma}[u]
      & := \frac{1}{2}\iint_{\R_+\times \R_+} \left|u(r) - u(s)\right|^2 \, \nu_{\frac{d-1}{2}+\ell}(r,s)\, h(r)h(s)\, (rs)^{d-1+2\ell}\,dr\,ds.
    \end{align}
    For $\alpha=2$, we let
    \begin{align}
      \ci_{\ell,\sigma}[u]
      & := \int_{\R_+}|u'(r)|^2h(r)^2r^{d-1+2\ell}dr.
    \end{align}
  \end{subequations}\index{$\ci_{\ell,\sigma}$}Then,
  \begin{align}
    \label{eq:gstransformreformulatedcor}
    \ce[[h\cdot u]_{\ell,m}]
    & = \ci_{\ell,\sigma}[u]
      + \Phi_\ell(\sigma)\int_{\R_+} \frac{|h(r) \cdot u(r)|^2}{r^\alpha}\,r^{d-1+2\ell}\,dr.
  \end{align}
\end{theorem}

\begin{remarks}
  \label{remarksmainresult}
  (1) Since the integral on the right-hand side of \eqref{eq:gstransformrereformulated} vanishes if $u$ is constant, by considering \eqref{eq:gstransformreformulatedcor}, we call the function
  \begin{align}
    \label{eq:defcapitalh}
    H(x) := [h]_{\ell,m}(x) = |x|^{-\sigma} Y_{\ell,m}(\omega_x)
  \end{align}
  a generalized ground state for $|D|^\alpha-\Phi_\ell(\sigma)|x|^{-\alpha}$ restricted to $V_\ell$, cf.~the discussion of generalized ground states after \eqref{eq:orggstransformtogether}.
  \\
  (2) The right-hand side of \eqref{eq:gstransformreformulatedcor} is the sum of two forms on $\R_+$,  but the isomorphism $[\cdot]_{\ell,m}$ converts \eqref{eq:gstransformreformulatedcor} into a statement about forms on $V_{\ell,m}$.
  Indeed, by \eqref{eq:defeqclassellm},
  \begin{align}
    \label{eq:potentialrdrplusformulations}
    \int_{\R_+} \frac{|h(r) \cdot u(r)|^2}{r^\alpha}\,r^{d-1+2\ell}\,dr
    = \int_{\R^d} \frac{|[h \cdot u]_{\ell,m}(x)|^2}{|x|^\alpha}\,dx.
  \end{align}
  Moreover, by \eqref{eq:nuellmainthm}, we have for $\alpha<2$,
  \begin{align}
    \label{eq:gstransformrereformulated}
    \ci_{\ell,\sigma}[u]
    = \frac12\iint\limits_{\R^d\times\R^d} \nu(x,y) \cdot \left| u(|x|) - u(|y|)\right|^2 \, \overline{H(x)}\,H(y)\,dx\,dy,
  \end{align}
  and the integral converges absolutely since $2\sigma<d$ and $u\in C_c^\infty(\R_+)$. If $\alpha=2$, then
  \begin{align}
    \ci_{\ell,\sigma}[u] = \int_{\R^d}|\nabla u(|x|)|^2 \cdot \left|H(x)\right|^2\,dx.
  \end{align}
  \\
  (3) Proposition~\ref{potentialexpansionsphericalharmonics} contains an explicit formula and sharp bounds for $\nu_{\frac{d-1}{2}+\ell}(r,s)$. Interestingly, by \cite[(15.4.7), (15.4.9)]{NIST:DLMF},
  \begin{subequations}
    \begin{align}
      \nu_{0}(r,s) & = \ca_{1,-\alpha} \left(|r-s|^{-1-\alpha} + (r+s)^{-1-\alpha}\right) \quad \text{and} \\
      \nu_{1}(r,s) & = \ca_{3,-\alpha} \cdot \frac{2\pi}{1+\alpha}\cdot\frac{|r-s|^{-1-\alpha}-(r+s)^{-1-\alpha}}{rs}.
    \end{align}
  \end{subequations}
  Note that
  $(d-1)/2+\ell=0$ corresponds to
  $d=1$ with $\ell=0$ (subordinate Brownian motion on symmetric functions) and
  $(d-1)/2+\ell=1$ corresponds to
  $d=1$ with $\ell=1$ (subordinate Brownian motion on antisymmetric functions) \emph{and}
  $d=3$ with $\ell=0$ (three-dimensional subordinate Brownian motion on radial functions).
  \\
  (4) When $\sigma=-\ell$, we have $\ce[[u]_{\ell,m}]=\ci_{\ell,-\ell}[u]$; cf.~\eqref{eq:ezetalimit} and Proposition~\ref{relationformheatkernel}.
  \\
  (5) By a density argument, \eqref{eq:defiellsigma} and \eqref{eq:gstransformreformulatedcor} extend to all functions $u$ on $\R_+$ for which we have $[h\cdot u]_{\ell,m}\in\dot H^{\alpha/2}(\R^d)$, provided $\alpha<d$. If $r^{\frac{d-1}{2}-\sigma}u(r) \in L^2(\R_+,dr)$, this is equivalent to $\int_0^\infty k^\alpha \left|\F_\ell\left(r^{\frac{d-1}{2}-\sigma}u\right)\right|^2\,dk<\infty$, with the Fourier--Bessel transform $\F_\ell$ introduced in~\eqref{eq:deffourierbessel}. See also the paragraph following~\eqref{eq:orggstransformtogether} and Lemma~\ref{fourierbessel} below.
  \\
  (6) It transpires from the proof of Theorem~\ref{gstransformreformulated} that \eqref{eq:gstransformreformulatedcor} for $\alpha<2$ also holds for $u\in L^2(\R_+,r^{d-1+2\ell}dr)$, see Theorem~\ref{hbetagammatransformed} and \cite[Proposition~5]{Bogdanetal2016}.
  \\
  (7) Jakubowski and Maciocha \cite{JakubowskiMaciocha2022} recently proved a ground state representation for the fractional Laplacian with Hardy potential on $\R_+$ with Dirichlet boundary condition on $(\R_+)^c$. As in our analysis, they use the abstract results of \cite{Bogdanetal2016}, but, in contrast to our approach, they cannot rely on subordination since the Dirichlet fractional Laplacian is only form-bounded from above by the fractional Dirichlet Laplacian, see, e.g., \cite[Lemma~6.4]{FrankGeisinger2016}. In particular, the results of \cite{JakubowskiMaciocha2022} are inequivalent to ours as can be readily seen by comparing the parameterizations of the coupling constants in \cite[(1.8)]{JakubowskiMaciocha2022} and ours in \eqref{eq:defsigmagammal}.
\end{remarks}

The proof of Theorem~\ref{gstransformreformulated} is given in Section~\ref{s:proofsgstransformreformulated}.
In view of Lemma~\ref{fourierbesselcor}, Theorem~\ref{gstransformreformulated} yields the following corollary, an \emph{angular momentum synthesis}.

\begin{corollary}
  \label{gstransformreformulatedalll}
  Let $d\in\N$, $\alpha\in(0,2]$, and $\ci$ be as in \eqref{eq:defiellsigma}.
  Let $f\in C_c^\infty(\R^d\setminus\{0\})$ have the $L^2$-expansion \eqref{eq:l2expansionsphericalharmonics} with coefficients $f_{\ell,m}=[h_{\ell,m}\cdot u_{\ell,m}]_{\ell,m}$, where $\ell\in L_d$, $m\in M_\ell$, $\alpha\in(0,d+2\ell)$, $\sigma_{\ell,m}\in[-\ell,(d-\alpha)/2]$, and $h_{\ell,m}(r):=r^{-\ell-\sigma_{\ell,m}}$. Then
  \begin{align}
    \label{eq:gstransformreformulatedalll}
    \begin{split}
      \ce[f]
      & = \sum_{\ell\in L_d, m\in M_\ell}\ce[f_{\ell,m}] \\
        & = \sum_{\ell\in L_d, m\in M_\ell} \left(\ci_{\ell,\sigma_{\ell,m}}[u_{\ell,m}] + \Phi_\ell(\sigma_{\ell,m})\int_{\R^d} dx\, \frac{|f_{\ell,m}(x)|^2}{|x|^\alpha}\right).
    \end{split}
  \end{align}
\end{corollary}

\begin{remark}
  Corollary~\ref{gstransformreformulatedalll} remains valid for all $f\in \dot H^{\alpha/2}(\R^d)$ when $\alpha<d$. If $\alpha<2$, then it also holds for all $f\in L^2(\R^d)$. Compare with Remarks~\ref{remarksmainresult}(4) and (5).
\end{remark}

\smallskip
The rest of this Introduction is structured as follows. In the next three subsections, we motivate Theorem \ref{gstransformreformulated}. Thus, in Section~\ref{s:yafaev}, we emphasize that \eqref{eq:hardy} even holds for $\kappa$ in the larger interval $(-\infty,\kappa_{\rm c}(\ell)]$, if, in addition, $f\in \bigoplus_{\ell'\geq\ell}V_{\ell'}$.
In Sections~\ref{s:discussion}, we compare the ground state representations in Theorem~\ref{gstransformreformulated} with those in \eqref{eq:orggstransform}--\eqref{e.gs2}.
In Section~\ref{s:applications}, we indicate possible applications of Theorem~\ref{gstransformreformulated}.
In Section~\ref{s:ideaproof}, we outline our proof of Theorem~\ref{gstransformreformulated}.
The Introduction ends with Section~\ref{s:organization}, where we give an overview of the rest of the paper and collect commonly used notation.

\subsection{Improved Hardy inequality functions with angular oscillations}
\label{s:yafaev}

In view of the previous discussions, it is natural to consider \eqref{eq:hardy} for functions in $V_{\ell,m}$. For $\ell=0$ such functions are radial, but they have angular oscillatory behavior for $\ell>0$, which leads to an improvement of \eqref{eq:hardy}. This is easily seen when $\alpha=2$. Recall that $\{Y_{\ell,m}\}_{\ell,m}$ are eigenfunctions of the Laplace--Beltrami operator $-\Delta_{\bs^{d-1}}$, with eigenvalues $\ell(\ell+d-2)$. By the representation of the Laplacian in spherical coordinates and the original inequality of Hardy, i.e., $\int_0^\infty |v'(r)|^2\,dr\geq4^{-1}\int_0^\infty |v(r)|^2r^{-2}\,dr$, we get
\begin{align}
  \label{eq:yafaev2}
  \begin{split}
    & \langle[u]_{\ell,m},(-\Delta)[u]_{\ell,m}\rangle_{L^2(\R^d)} \\
    & \quad = \left\langle Uu,\left(-\frac{d^2}{dr^2} + \frac{(d-1)(d-3)+4\ell(\ell+d-2)}{4r^2}\right)Uu\right\rangle_{L^2(\R_+,dr)} \\
    & \quad \geq \kappa_{\rm c}(\ell) \int_0^\infty \frac{|(Uu)(r)|^2}{r^2}\,dr
      = \kappa_{\rm c}(\ell) \int_{\R^d}\frac{|[u]_{\ell,m}(x)|^2}{|x|^2}\,dx.
  \end{split}
\end{align}
Here $U:L^2(\R_+,r^{d-1+2\ell}dr)\to L^2(\R_+,dr)$ is the unitary operator
\begin{align}
  \label{eq:defdoob}
  L^2(\R_+,r^{d-1+2\ell}dr) \ni u \mapsto (Uu)(r)=r^{(d-1+2\ell)/2}u(r) \in L^2(\R_+,dr).
\end{align}
More generally, for $\ell\in L_d$ and $\alpha\in(0,d)$, the improved Hardy inequality
\begin{align}
  \label{eq:yafaev3}
  \ce[[u]_{\ell,m}]
  \geq \kappa\int_{\R^d}\frac{|[u]_{\ell,m}(x)|^2}{|x|^{\alpha}}\,dx
\end{align}
is known to hold for all $u\in C_c^\infty(\R_+)$ if and only if $\kappa\leq\kappa_{\rm c}(\ell)$. 
While \eqref{eq:yafaev2} and \eqref{eq:yafaev3} for $\ell=0$ merely confirm \eqref{eq:hardy} with $\kappa\leq\kappa_c$ and radial functions, they are stronger when $\ell>0$ since $\kappa_c(\ell)>\kappa_c$ (see Proposition~\ref{propertiesphi}).
For $d=3$ and $\alpha=1$, Inequality \eqref{eq:yafaev3} was first shown by Le Yaouanc, Oliver, and Raynal \cite[Section~VI]{LeYaouancetal1997}, while the general case was proved by Yafaev \cite[(2.4), (2.26)]{Yafaev1999}. We note that his proof works for all $\alpha\in(0,d)$. Since our proof of Theorem~\ref{gstransformreformulated} is based on the subordination of $\me{t\Delta}$, we restrict our considerations to $\alpha\in(0,2]$, in fact to $\alpha\in(0,2]\cap(0,d+2\ell)$.
Incidentally, for $d\in\{1,2\}$ and $\ell=1$, this interval contains $\alpha\in(0,d)$.

\subsection{Comparison between the ground state representations}
\label{s:discussion}

Let us compare the ground state representations in \eqref{eq:orggstransform}--\eqref{e.gs2} with those in Theorem~\ref{gstransformreformulated} and Corollary~\ref{gstransformreformulatedalll}. For brevity, we only consider $\alpha\in(0,2)$.
\\
(1) Both representations agree for $\ell=0$.
Indeed, suppose $\sigma\in[0,(d-\alpha)/2]$, $u\in C_c^\infty(\R_+)$, and $f(x)=|\bs^{d-1}|^{-1/2}u(|x|)|x|^{-\sigma}$. By Theorem~\ref{gstransformreformulated},
\begin{align*}
  \ce[f] = \Phi_0(\sigma)\int_{\R^d}\frac{|f(x)|^2}{|x|^\alpha}\,dx + \frac12 \int_{\R_+^2}dr\,ds\, (rs)^{d-1-\sigma}\nu_{\frac{d-1}{2}}(r,s)|u(r)-u(s)|^2,
\end{align*}
where, indeed, by \eqref{eq:nuellmainthm}, the rightmost integral equals the rightmost integral in \eqref{eq:orggstransform}. Therefore, our principal interest is in the analysis of $\ell>0$.
\\
(2) Suppose $f=\sum_{\ell,m}f_{\ell,m}$, as in \eqref{eq:l2expansionsphericalharmonics}. Then both the ground state representations in Corollary~\ref{gstransformreformulatedalll} and, by Lemma~\ref{fourierbesselcor}, those in \eqref{eq:orggstransformtogether} break down into ground state representations corresponding to different spaces $V_{\ell,m}$.
To compare the representations of $\ce[f]$ by \eqref{eq:orggstransform} or \eqref{eq:gstransformreformulatedcor} and \eqref{eq:gstransformreformulatedalll} in more detail, we record that
\begin{align*}
  (rs)^{-\ell}\int_{\bs^{d-1}\times\bs^{d-1}}d\omega_x\,d\omega_y\, \overline{Y_{\ell,m}(\omega_x)}Y_{\ell',m'}(\omega_y)\nu(r\omega_x-s\omega_y)
  = \nu_{\frac{d-1}{2}+\ell}(r,s)\delta_{\ell,\ell'}\delta_{m,m'}
\end{align*}
holds for $\ell,\ell'\in L_d$, $m\in M_\ell$, and $m'\in M_{\ell'}$ by the expression of $\nu_{\frac{d-1}{2}+\ell}$ in Theorem~\ref{gstransformreformulated} and the spherical symmetry; see Corollary~\ref{kalfcor}. Here, $\delta_{\ell,\ell'}$ denotes the Kronecker delta. Therefore, if
$f_{\ell,m}(x)=u_{\ell,m}(|x|)|x|^{-\sigma}Y_{\ell,m}(\omega_x)$,
then the ground state representation \eqref{eq:orggstransform} is equivalent to
\begin{align}
  \label{eq:orggstransformanglesintegrated}
  & \ce[f] - \Phi(\sigma)\int_{\R^d}|f(x)|^2\,|x|^{-\alpha}\,dx \\
  \begin{split}
    \notag
    & \ = \sum_{\ell,m;\ell',m'}\int_{\R_+\times\R_+}dr\,ds\, (rs)^{d-1-\sigma} \int_{\bs^{d-1}\times\bs^{d-1}}d\omega_x\,d\omega_y\, \nu(r\omega_x-s\omega_y) \\
    & \quad \times \left[\overline{u_{\ell,m}(r)}u_{\ell',m'}(r)\overline{Y_{\ell,m}(\omega_x)}\,Y_{\ell',m'}(\omega_x) - \re\left(\overline{u_{\ell,m}(r)}u_{\ell',m'}(s)\overline{Y_{\ell,m}(\omega_x)}\,Y_{\ell',m'}(\omega_y)\right)\right] \\
    & \ = \sum_{\ell,m}\int_{\R_+\times\R_+}\!\! dr\,ds\, (rs)^{d-1-\sigma}\left[|u_{\ell,m}(r)|^2\nu_{\frac{d-1}{2}}(r,s) \!-\! \re\!\left(\overline{u_{\ell,m}(r)}\,u_{\ell,m}(s)\right)\!\nu_{\frac{d-1}{2}+\ell}(r,s)\right].
  \end{split}
\end{align}
We see the following two differences between \eqref{eq:orggstransform} and \eqref{eq:gstransformreformulatedcor} for any fixed $\ell>0$:

\begin{enumerate}
\item[(a)] The Hardy potential has a prefactor of $\Phi(\sigma)$ in \eqref{eq:orggstransform}, while it has a prefactor of $\Phi_\ell(\sigma)$ in \eqref{eq:gstransformreformulatedcor}, which is greater than $\Phi(\sigma)$ (see Proposition~\ref{propertiesphi}). In particular, for functions $f\in C_c^\infty\cap V_{\ell}$ with fixed $\ell\in L_d$, Theorem~\ref{gstransformreformulated} yields the Hardy inequality \eqref{eq:hardy} with the larger, and optimal, constant $\kappa_{\rm c}(\ell)$.
  
\item[(b)] As the last line of \eqref{eq:orggstransformanglesintegrated} reveals, the integrand in the ground state representation \eqref{eq:orggstransform} is
  \begin{align}
    \label{eq:orggstransformkernelcomparison}
    (rs)^{d-1-\sigma}\left[|u_{\ell,m}(r)|^2\nu_{\frac{d-1}{2}}(r,s) - \re\left(\overline{u_{\ell,m}(r)}\,u_{\ell,m}(s)\right)\nu_{\frac{d-1}{2}+\ell}(r,s)\right].
  \end{align}
  In particular, the term $|u_{\ell,m}(r)|^2$ is multiplied by $\nu_{\frac{d-1}{2}}(r,s)$, while that involving $\overline{u_{\ell,m}(r)}\,u_{\ell,m}(s)$ is multiplied by $\nu_{\frac{d-1}{2}+\ell}(r,s)$ which is different and prevents us from completing the squares. On the other hand, after expanding the square, the integrand in the ground state representation \eqref{eq:gstransformreformulatedcor} is
  \begin{align}
    \label{eq:newgstransformkernelcomparison}
    (rs)^{d-1-\sigma} \nu_{\frac{d-1}{2}+\ell}(r,s) \left[|u_{\ell,m}(r)|^2-\re\left(\overline{u_{\ell,m}(r)}\,u_{\ell,m}(s)\right)\right] h(r)h(s).
  \end{align}
  Thus, in contrast to the expression in \eqref{eq:orggstransformkernelcomparison}, also the term $|u_{\ell,m}(r)|^2$ is multiplied by $\nu_{\frac{d-1}{2}+\ell}(r,s)$, and not by $\nu_{\frac{d-1}{2}}(r,s)$, so we can complete squares.

\end{enumerate}

In summary, Theorem~\ref{gstransformreformulated} and Corollary~\ref{gstransformreformulatedalll} can be seen as refinements of the ground state representations \eqref{eq:orggstransform}--\eqref{e.gs2} since they take into account angular oscillations.

\subsection{Implications of Theorem~\ref{gstransformreformulated}}
\label{s:applications}

Let us comment on a broader impact of Theorem~\ref{gstransformreformulated} for $\alpha<2$.
Bogdan, Grzywny, Jakubowski, and Pilarczyk \cite{Bogdanetal2019} proved that, for $h(r)=r^{-\sigma}$, $\exp\left(-t(|D|^\alpha-\Phi_0(\sigma)/|x|^\alpha)\right)[h]_{0,0}(x)=[h]_{0,0}(x)$, which
further motivates calling $[h]_{0,0}$ a ground state. This invariance was crucial to prove pointwise bounds for $\exp(-t(|D|^\alpha-\Phi_0(\sigma)/|x|^\alpha))$. In turn, these bounds were fundamental for analyzing the $L^p$-Sobolev spaces generated by powers of $|D|^\alpha-\Phi_0(\sigma)/|x|^\alpha$, see \cite{Franketal2021,Merz2021,BuiDAncona2023,BuiNader2022}, and for the proof of the strong Scott conjecture for relativistic atoms \cite{Franketal2020P,Franketal2023,Franketal2023T}. The latter works concern the analysis of the electron density of a large atom described by a multi-particle analog of Chandrasekhar's operator $\sqrt{-\Delta+1}-1-\kappa/|x|$ in $L^2(\R^3)$. There, the heat kernel bounds of \cite{Bogdanetal2019} were also used to give upper bounds for the sums of squares of the eigenfunctions of the Chandrasekhar operator. On the other hand, Jakubowski, Kaleta, and Szczypkowski \cite{Jakubowskietal2022} used heat kernel bounds to investigate single eigenfunctions of the Chandrasekhar operator and, in particular, proved lower bounds for the ground state. As in these works, the results of the present paper give a starting point for a deeper analysis of $\exp(-t(|D|^\alpha-\Phi_\ell(\sigma)/|x|^\alpha))$, powers of $|D|^\alpha-\Phi_\ell(\sigma)/|x|^\alpha$, and eigenfunctions and their sums for the Chandrasekhar operator projected onto $V_{\ell,m}$.

\subsection{Idea of the proof of Theorem~\ref{gstransformreformulated}}
\label{s:ideaproof}

We will use the results of Bogdan, Dyda, and Kim~\cite[Sections~2--3]{Bogdanetal2016} to study $|D|^\alpha -\Phi_\ell(\sigma)|x|^{-\alpha}$ in $V_{\ell,m}$. In that paper, the authors consider a symmetric sub-probability transition density $p_t(x,y)$ on some $\sigma$-finite measure space $(X,\cm,m)$ and the quadratic form $\gE[u]=\lim_{t\searrow0}t^{-1}\langle u,(1-p_t)u\rangle$, $u\in L^2(X,m)$.
They construct positive, supermedian functions $h$ by integrating $p_t$ against suitable functions in space and time, and use $h$ to derive a Hardy inequality and a ground state representation for $\gE$.
Formally, the ground state $|x|^{-\sigma}Y_{\ell,m}(\omega_x)$ of $|D|^\alpha-\Phi_\ell(\sigma)|x|^{-\alpha}$ is an analogue of $h$, but it is an oscillating function. However, since $V_{\ell,m}=L^2(\R_+,r^{d-1+2\ell}dr)\otimes\spann\{Y_{\ell,m}\}$, we see that the positive function $r^{-\ell-\sigma}$ is the radial part of the ground state of the restriction of $|D|^\alpha-\Phi_\ell(\sigma)|x|^{-\alpha}$ to $V_{\ell,m}$. Thus, we use \cite{Bogdanetal2016} to find ground states of and a ground state representation for the radial part of $|D|^\alpha-\Phi_\ell(\sigma)|x|^{-\alpha}$ restricted to $V_{\ell,m}$.

Specifically, our approach to prove Theorem~\ref{gstransformreformulated} is as follows.
First, we use the semigroup representation \eqref{eq:fractlaplaceheatkernel}, integrate over the unit spheres, and obtain the identity
\begin{align}
  \label{eq:besselemerges}
  \langle[u]_{\ell,m},P^{(\alpha)}(t,\cdot,\cdot)[u]_{\ell,m}\rangle_{L^2(\R^d)}
  = \langle u,p_{(d-1+2\ell)/2}^{(\alpha)}(t,\cdot,\cdot)u\rangle_{L^2(\R_+,r^{d-1+2\ell}dr)},
\end{align}
see Proposition~\ref{relationformheatkernel}. Here $p_\zeta^{(2)}(t,r,s)$ is the Bessel heat kernel of order $\zeta-1/2$ for any $\zeta\in(-1/2,\infty)$ and, when $\alpha\in(0,2)$, the kernels $p_\zeta^{(\alpha)}(t,r,s)$ are defined by the usual $\alpha/2$-stable subordination of $p_\zeta^{(2)}(t,r,s)$, see Definitions~\ref{defsemigroupptwise}--\ref{defsemigroupptwisealpha}.
We prove that $p_\zeta^{(\alpha)}(t,r,s)$ is a probability transition density (Section~\ref{s:semigroupproperties}) and compute the corresponding L\'evy kernel $\nu_{\zeta}(r,s):=\lim_{t\searrow0}t^{-1} p_\zeta^{(\alpha)}(t,r,s)$ (Section~\ref{s:kinenheatkernel}). Using the heat and L\'evy kernels, we define corresponding quadratic forms on $\R_+$, denoted by $\ce_\zeta$ (Section~\ref{s:defquadformsemigroup}). It is to these quadratic forms that we apply the abstract results of \cite{Bogdanetal2016} and prove a ground state representation (Theorem~\ref{hbetagammatransformed}). Finally, using \eqref{eq:besselemerges}, we show that $\ce[[u]_{\ell,m}]=\ce_\zeta[u]$ (Proposition~\ref{relationformheatkernel}).
This allows us to conclude the proof of Theorem~\ref{gstransformreformulated}.

\smallskip
Let us also note that there are further methods, similar in spirit, to construct ground states (or supersolutions) and Hardy (in)equalities, e.g., by Devyver, Fraas, and Pinchover \cite{Devyveretal2014}. Their techniques work well for second-order elliptic operators. It is an interesting problem to see if their approach can be adapted to treat also other operators, e.g., those that we discuss in the present paper.
We note that Dyda \cite[p.~550]{Dyda2012} offers an approach to Hardy inequalities exploiting signed (antisymmetric) functions, which is much in spirit of our development, at least for $\ell=1$. We refer to Kulczycki and Ryznar \cite{KulczyckiRyznar2016} for a related notion of \emph{difference process}, too.
We were also inspired by intertwining of operators, studied, e.g., by Patie and Savov \cite{PatieSavov2021}, which motivates a philosophy to escape the restriction of non-negativity when studying the semigroup $P^{(\alpha)}$ and the form $\ce^{(\alpha)}$ on $V_{\ell,m}$.

\subsection{Organization and notation}
\label{s:organization}

Here is the structure of the paper.
In Section~\ref{s:pointwise}, we define and analyze the Bessel heat kernels and their $\alpha/2$-subordinated kernels. Moreover, we study the associated L\'evy kernels and use them to define the corresponding quadratic forms on $\R_+$.
In Section~\ref{s:optimalhardyalpha2}, we prove a ground state representation for the quadratic forms defined in Section~\ref{s:pointwise} (Theorem~\ref{hbetagammatransformed}).
In Section~\ref{s:notation}, we study $|D|^\alpha$ and $\me{-t|D|^\alpha}$ when restricted to $ V_{\ell,m}$ and establish the link to the quadratic forms and semigroups defined in Section~\ref{s:pointwise}.
With this connection at hand, in Section~\ref{s:proofsgstransformreformulated}, we apply Theorem~\ref{hbetagammatransformed} to prove Theorem~\ref{gstransformreformulated}.
The Appendix contains auxiliary results.

\smallskip
Below we denote generic (real positive) constants by $c\in(0,\infty)$. The values of constants may change from place to place. We mark the dependence of $c$ on some parameter $\tau$ by $c_\tau$. We write $A\lesssim B$ for functions $A,B\geq0$ to indicate that there is a constant $c$ such that $A\leq c B$. If $c$ depends on $\tau$, we may write $A\lesssim_\tau B$, but the dependence on $d$ and $\alpha$ is usually ignored in the notation. The notation $A\sim B$ means that $A\lesssim B\lesssim A$ and then we say \emph{$A$ is comparable, or, equivalent to $B$}.\index{$\lesssim$}\index{$\sim$} We abbreviate $A\wedge B:=\min\{A,B\}$ and $A\vee B:=\max\{A,B\}$.\index{$\wedge$}\index{$\vee$} For $x\in\R$, the Gauss bracket $\floor{x}$ denotes the integer part of~$x\in\R$.\index{$\floor{x}$}
The function spaces, e.g., $L^2(\R^d)$, are complex.
For $z\in\C\setminus(-\infty,0]$ we denote the ordinary and modified Bessel functions of the first kind of order $\nu\in\C$ by $J_\nu(z)$ and $I_\nu(z)$, respectively \cite[(10.2.2), (10.25.2)]{NIST:DLMF}.\index{$J_\nu$}\index{$I_\nu$} We write $_2\tilde F_1(a,b;c;z) =\, _2F_1(a,b;c;z)/\Gamma(c)$, with $a,b,c\in\C$ and $z\in\{w\in\C:\,|w|<1\}$, for the regularized hypergeometric function \cite[(15.2.1)]{NIST:DLMF} and $_2F_1(a,b;c;z)$ for the (usual) hypergeometric function \cite[(15.2.2)]{NIST:DLMF} when $c\notin\{0,-1,-2,...\}$.\index{$_2F_1(a,b;c;z)$}\index{$_2\tilde F_1(a,b;c;z)$} Moreover, $_1\tilde F_1(a;b;z)\!=\!\, _1F_1(a;b;z)/\Gamma(b)$, with $a,b,z\!\in\!\C$, denotes the regularized confluent hypergeometric function, where $_1F_1(a;b;z)$ denotes the Kummer confluent hypergeometric function \cite[(13.2.2)]{NIST:DLMF}, defined for $b\notin\{0,-1,-2,...\}$.\index{$_1F_1(a;b;z)$}\index{$_1\tilde F_1(a;b;z)$} We introduce further notation as we proceed. The notation is summarized in the \hyperref[index]{\indexname} at the end of this paper.

\section{Subordinated Bessel kernels}
\label{s:pointwise}

\subsection{Properties of heat kernels}
\label{s:semigroupproperties}

Our goal is to investigate the Gaussian heat kernel
\begin{align}
  \me{t\Delta}(x,y)
  = P^{(2)}(t,x,y)
  = \frac{1}{(4\pi t)^{d/2}}\exp\left(-\frac{|x-y|^2}{4t}\right), \quad t>0,\ x,y\in\R^d
\end{align}
and its $\frac\alpha2$-subordinated semigroups $\me{-t|D|^\alpha}$, restricted to $V_{\ell,m}$ for each $\ell\in L_d$, $m\in M_\ell$. According to Proposition~\ref{relationformheatkernel} below, this leads to kernels defined on $\R_+$, which depend on $\alpha\in(0,2]$ and $(d-1)/2+\ell$ with $d\in\N$ and $\ell\in L_d$, but not on $m\in M_\ell$. For the sake of future reference, we consider general kernels \emph{indexed} by $\alpha\in(0,2]$ and an \emph{arbitrary} parameter $\zeta\in(-1/2,\infty)$. In the context of Theorem~\ref{gstransformreformulated}, we may think of $\zeta$ as equal to $(d-1)/2+\ell$. We first define and analyze these kernels for $\alpha=2$. Afterwards, we use subordination to study the case $\alpha<2$.

\begin{definition}
  \label{defsemigroupptwise}
  Let $\zeta\in(-1/2,\infty)$.
  We consider the reference (speed) measure $r^{2\zeta}dr$ on $\R_+$ and, for $r,s,t>0$, define
  \begin{align}
    \label{eq:defpheatalpha2}
    \begin{split}
      p_\zeta^{(2)}(t,r,s) & : = \frac{(rs)^{1/2-\zeta}}{2t}\exp\left(-\frac{r^2+s^2}{4t}\right)I_{\zeta-1/2}\left(\frac{rs}{2t}\right),
    \end{split}
  \end{align}
  where $I_{\zeta-1/2}$ is the modified Bessel function of the first kind.
\end{definition}
The kernel $p_\zeta^{(2)}$ is the heat kernel associated to the Bessel operator of order $\zeta-1/2$, namely to $-\frac{d^2}{dx^2}-2\zeta x^{-1}\frac{d}{dx}$ in $L^2(\R_+,r^{2\zeta}dr)$.
The notation for Bessel kernels varies in the literature. Our notation is close, e.g., to \cite[Section~2.2]{Maleckietal2016}, with $\nu$ and $t$ therein equal to $\zeta-1/2$ and $2t$ here, respectively.

Of essential importance for us is the fact that $p_\zeta^{(2)}(t,r,s)$ is a probability transition density. That is, for all $\zeta\in(-1/2,\infty)$, $t,t',r,s>0$,
\begin{enumerate}
\item $p_\zeta^{(2)}(t,r,s)>0$, 
\item $\int_0^\infty p_\zeta^{(2)}(t,r,s)\,s^{2\zeta}ds = 1$, and 
\item $\int_0^\infty p_\zeta^{(2)}(t,r,z) p_\zeta^{(2)}(t',z,s) z^{2\zeta}\,dz = p_\zeta^{(2)}(t+t',r,s)$. 
\end{enumerate}
The positivity is immediate from the positivity of the modified Bessel function, see \cite[(10.25.2)]{NIST:DLMF}. For convenience, we provide computations proving the other two properties (normalization and Chapman--Kolmogorov equations) in the Appendices~\ref{s:normalized} and \ref{s:chapman}. We also note the scaling
\begin{align}
  \label{eq:scalingalpha2}
  p_\zeta^{(2)}(t,r,s)
  = t^{-\frac{2\zeta+1}{2}} p_\zeta^{(2)}\left(1,\frac{r}{t^{1/2}},\frac{s}{t^{1/2}}\right),
\end{align}
which is immediate from \eqref{eq:defpheatalpha2}.
For a much more detailed analysis of $p_\zeta^{(2)}$, we refer, e.g., to the textbooks \cite[Part I, Section~IV.6 or Appendix~1.21]{BorodinSalminen2002} or \cite[Chapter~XI]{RevuzYor1999}.

\smallskip
Now we subordinate $p_\zeta^{(2)}$. Recall that for $\alpha\in(0,2)$ and $t>0$, by Bernstein's theorem, the completely monotone function $[0,\infty)\ni\lambda\mapsto\me{-t\lambda^{\alpha/2}}$  is the Laplace transform of a probability density function $\R_+\ni\tau\mapsto\sigma_t^{(\alpha/2)}(\tau)$. That is,
\begin{align}
  \label{eq:subordination}
  \me{-t\lambda^{\alpha/2}} = \int_0^\infty \me{-\tau\lambda}\,\sigma_t^{(\alpha/2)}(\tau)\,d\tau, \quad t>0,\,\lambda\geq0,
\end{align}
see, e.g., \cite[Chapter~5]{Schillingetal2012}.
We record some useful properties of $\sigma_t^{(\alpha/2)}(\tau)$ in Appendix~\ref{s:propertiessubordinator}. We now introduce the $\tfrac\alpha2$-subordinated Bessel kernels.

\begin{definition}
  \label{defsemigroupptwisealpha}
  Let $\zeta\in(-1/2,\infty)$ and $\alpha\in(0,2)$.
  We consider the reference measure $r^{2\zeta}dr$ on $\R_+$ and, for $r,s,t>0$, define
  \begin{align}
    \label{eq:defpheatalpha}
    \begin{split}
      p_\zeta^{(\alpha)}(t,r,s) & : = \int_0^\infty p_\zeta^{(2)}(\tau,r,s)\,\sigma_t^{(\alpha/2)}(\tau)\,d\tau.
    \end{split}
  \end{align}\index{$p_\zeta^{(\alpha)}(t,r,s)$}
\end{definition}

Note that the integral operators with kernel $p_\zeta^{(\alpha)}(t,\cdot,\cdot)$, $\alpha\in(0,2]$, $t>0$, act with respect to the measure $r^{2\zeta}dr$, i.e.,
$$
(p_\zeta^{(\alpha)}(t,\cdot,\cdot)u)(r) = \int_0^\infty p_\zeta^{(\alpha)}(t,r,s)u(s) s^{2\zeta}\,ds.
$$

Since $\sigma_t^{(\alpha/2)}$ is a convolution semigroup of probability transition density functions \cite[p.~48--49]{Schillingetal2012} (see also \eqref{eq:subordinatornormalized} and \eqref{eq:subordinatorconvolution} in Appendix~\ref{s:propertiessubordinator}), the above analysis for $\alpha=2$ implies that $p_\zeta^{(\alpha)}(t,\cdot,\cdot)$ is a probability transition density for all $\zeta\in(-1/2,\infty)$ and $\alpha\in(0,2]$. We summarize this in the following proposition.

\begin{proposition}
  \label{summarypropertiespzeta}
  Let $\zeta\in(-1/2,\infty)$, $\alpha\in(0,2]$, and $t,t',r,s>0$. Then we have $p_\zeta^{(\alpha)}(t,r,s)>0$ and
  \begin{align}
    \label{eq:normalizedalpha}
    & \int_0^\infty p_\zeta^{(\alpha)}(t,r,s) s^{2\zeta}\,ds = 1, \\
    \label{eq:chapman}
    & \int_0^\infty p_\zeta^{(\alpha)}(t,r,z) p_\zeta^{(\alpha)}(t',z,s) z^{2\zeta}\,dz = p_\zeta^{(\alpha)}(t+t',r,s), \\
    \label{eq:scalingalpha}
    & p_\zeta^{(\alpha)}(t,r,s) = t^{-\frac{2\zeta+1}{\alpha}} p_\zeta^{(\alpha)}\left(1,\frac{r}{t^{1/\alpha}},\frac{s}{t^{1/\alpha}}\right).
  \end{align}
\end{proposition}

\begin{proof}
  We discussed these claims already for $\alpha\!=\!2$, so let $\alpha\in(0,2)$.
  The positivity of $p_\zeta^{(\alpha)}(t,r,s)$ for $\alpha\!\in\!(0,2)$ follows from subordination \eqref{eq:defpheatalpha} and the positivity of $p_\zeta^{(2)}(t,r,s)$ and $\sigma_t^{(\alpha/2)}(\tau)$.
  The normalization \eqref{eq:normalizedalpha} for $\alpha\!\in\!(0,2)$ follows from subordination \eqref{eq:defpheatalpha}, the normalization when $\alpha\!=\!2$, Tonelli's theorem, and the fact that $\sigma_t^{(\alpha/2)}$ is a probability density.
  The Chapman--Kolmogorov equation \eqref{eq:chapman} for $\alpha\!\in\!(0,2)$ follows from subordination \eqref{eq:defpheatalpha} and the corresponding equation for $\alpha\!=\!2$ since
  \begin{align}
    \begin{split}
      & \int_0^\infty p_\zeta^{(\alpha)}(t,r,z) p_\zeta^{(\alpha)}(t',z,s)z^{2\zeta}\,dz \\
      & \quad = \int_0^\infty  d\tau \int_0^\infty d\tau'\, \sigma_t^{(\alpha/2)}(\tau)\, \sigma_{t'}^{(\alpha/2)}(\tau')\int_0^\infty p_\zeta^{(2)}(\tau,r,z) p_\zeta^{(2)}(\tau',z,s)z^{2\zeta}\,dz \\
      & \quad = \int_0^\infty d\tau \int_0^\infty d\tau'\, \sigma_t^{(\alpha/2)}(\tau)\, \sigma_{t'}^{(\alpha/2)}(\tau')\, p_\zeta^{(2)}(\tau+\tau',r,s)
      = p_\zeta^{(\alpha)}(t+t',r,s).
    \end{split}
  \end{align}
  In the last step, we used that $\sigma_t^{(\alpha/2)}(\tau)$ is a convolution semigroup.
  Finally, the scaling relation \eqref{eq:scalingalpha} for $\alpha\in(0,2)$ follows from the corresponding relation \eqref{eq:scalingalpha2} for $\alpha=2$ and the scaling of $\sigma_t^{(\alpha/2)}$ (see \eqref{eq:subordinatorscaling}). Indeed,
  \begin{align*}
    p_\zeta^{(\alpha)}(t,r,s)
    & = \int_0^\infty p_\zeta^{(2)}(\tau,r,s)\,\sigma_t^{(\frac\alpha2)}(\tau)\,d\tau
      = t^{-\frac{2\zeta+1}{\alpha}}\int_0^\infty p_\zeta^{(2)}\left(\frac{\tau}{t^{2/\alpha}},\frac{r}{t^{1/\alpha}},\frac{s}{t^{1/\alpha}}\right)\,\sigma_t^{(\frac\alpha2)}(\tau)\,d\tau \\
    & = t^{-\frac{2\zeta+1}{\alpha}}\int_0^\infty p_\zeta^{(2)}\left(\tau,\frac{r}{t^{1/\alpha}},\frac{s}{t^{1/\alpha}}\right)\,\sigma_1^{(\frac\alpha2)}(\tau)\,d\tau
      = t^{-\frac{2\zeta+1}{\alpha}} p_\zeta^{(\alpha)}\left(1,t^{-1/\alpha} r,t^{-1/\alpha} s\right).
  \end{align*}
\end{proof}

\subsection{Analysis of the integral kernel of $\ci_{\ell,\sigma}$}
\label{s:kinenheatkernel}

In the following, we explicitly compute the integral kernel $\nu_{\frac{d-1}{2}+\ell}(r,s)$ in the functional $\ci_{\ell,\sigma}$ in \eqref{eq:defiellsigmaalphal2} and provide sharp bounds for it.
To these ends, we recall and use the pointwise limit \eqref{eq:subordinatorquotientexplicit}, i.e.,
\begin{align}
  \label{eq:derivativesubordinator}
  \lim_{t\to0}\frac{\sigma_t^{(\alpha/2)}(\tau)}{t}
  = \frac{\Gamma(\alpha/2+1)\, \sin\left(\pi \alpha/2\right)}{\pi \, \tau^{1+\alpha/2}},
  \quad \tau>0,
\end{align}
and the uniform bound \eqref{eq:subordinatorbounds2}, i.e.,
\begin{align}
  \label{eq:uniformboundsubordinator}
  t^{-1} \sigma_t^{(\alpha/2)}(\tau) \lesssim_\alpha \tau^{-1-\alpha/2}, \quad t,\tau>0.
\end{align}

\begin{proposition}
  \label{potentialexpansionsphericalharmonics}
  Let $\zeta\in(-1/2,\infty)$, $\alpha\in(0,2)$, and $r,s>0$ with $r\neq s$. Then,
  \begin{align}
    \label{eq:defnuell1}
    \nu_{\zeta}(r,s)
    & := \lim_{t\searrow0}\frac{p_\zeta^{(\alpha)}(t,r,s)}{t}
      = \lim_{t\searrow0} \int_0^\infty p_\zeta^{(2)}(\tau,r,s)\frac{\sigma_t^{(\alpha/2)}(\tau)}{t}\,d\tau \\
    \label{eq:defnuell2}
    \begin{split}
      & \ = 2^{1+\alpha} \frac{\Gamma\left(\frac\alpha2+1\right)\,\sin\left(\frac{\pi\alpha}{2}\right)}{\pi}\, \Gamma\left(\zeta+\frac\alpha2+\frac12\right)\, (r^2+s^2)^{-(\zeta+\frac\alpha2+\frac12)} \\
      & \qquad \times\, _2\tilde F_1\left(\frac{\zeta+\alpha/2+1/2}{2}, \frac{\zeta+\alpha/2+3/2}{2}; \zeta+\frac12; \frac{4(rs)^2}{(r^2+s^2)^2}\right)
    \end{split} \\
    \label{eq:defnuell3}
    & \ \sim_{\zeta,\alpha} \frac{1}{(r+s)^{2\zeta} \cdot |r-s|^{1+\alpha}}.
  \end{align}\index{$\nu_\zeta(r,s)$}Moreover,
  \begin{align}
    \label{eq:defnuell4}
    \frac{p_\zeta^{(\alpha)}(t,r,s)}{t}
    \lesssim_{\zeta,\alpha} \nu_{\zeta}(r,s), \quad t>0,
  \end{align}
  and, for $\zeta\geq\zeta'>-1/2$,
  \begin{align}
    \label{eq:nuellmonotony}
    \nu_{\zeta}(r,s)\lesssim_{\zeta,\zeta',\alpha}\nu_{\zeta'}(r,s).
  \end{align}
\end{proposition}

\begin{proof}
  Formula~\eqref{eq:defnuell2} follows from dominated convergence using \eqref{eq:defpheatalpha2}, \eqref{eq:derivativesubordinator}, and \eqref{eq:uniformboundsubordinator}. The computation is contained in Section~\ref{s:selfdomconvsubordinatordensity}.
  In particular, \eqref{eq:defnuell4} follows.

  Now we prove \eqref{eq:defnuell3}. Without loss of generality, let $r>s$. First, we quote
  \begin{align}
    _2\tilde F_1\left(\frac a2,\frac{a+1}{2}; a-b+1;\frac{4z}{(1+z)^2}\right)
    = (1+z)^a \, _2\tilde F_1\left(a,b;a-b+1;z\right), \quad |z|<1,
  \end{align}
  with $a=\zeta+(\alpha+1)/2$, $b=(\alpha+2)/2$, and $z=(s/r)^2\in(0,1)$; see \cite[(15.8.15)]{NIST:DLMF}. Thus, with $c_{\zeta,\alpha}:= 2^{1+\alpha} \pi^{-1}\, \sin\left(\frac{\pi\alpha}{2}\right)\, \Gamma\left(\frac\alpha2+1\right)\, \Gamma\left(\zeta+\frac\alpha2+\frac12\right)$, we obtain
  \begin{align}
    \label{eq:defnuell3aux1}
    \nu_{\zeta}(r,s)
    = c_{\zeta,\alpha}\, (rs)^{-\zeta} \, \left(\frac{r\wedge s}{r\vee s}\right)^{\zeta} \, \frac{1}{(r\vee s)^{\alpha+1}} \, _2\tilde F_1\left(\zeta+\frac{\alpha+1}{2},\frac{\alpha+2}{2};\zeta+\frac12;\frac{s^2}{r^2}\right).
  \end{align}
  By distinguishing between $s/r\in(0,1/2]$ and $s/r\in[1/2,1)$, we have
  \begin{align}
    \label{eq:hypergeomboundclose1}
    _2\tilde F_1\left(\zeta+\frac{\alpha+1}{2},\frac{\alpha+2}{2};\zeta+\frac12;\frac{s^2}{r^2}\right) \sim_{\zeta,\alpha} |1-s^2/r^2|^{-1-\alpha}.
  \end{align}
  Indeed, for $s/r\leq1/2$, this follows from the definition of $_2\tilde F_1$ and for $s/r\in[1/2,1)$, we use \cite[(15.10.21)]{NIST:DLMF}, i.e.,
  \begin{align}
    \begin{split}
      & _2\tilde F_1\left(a,b;c,z\right)
      = \frac{\Gamma(c)\Gamma(c-a-b)}{\Gamma(c-a)\Gamma(c-b)} \cdot \,_2\tilde F_1\left(a,b;a+b-c+1,1-z\right) \\
      & \qquad + \frac{\Gamma(c)\Gamma(a+b-c) (1-z)^{c-a-b}}{\Gamma(a)\Gamma(b)} \cdot \,_2\tilde F_1\left(c-a,c-b;c-a-b+1,1-z\right)
    \end{split}
  \end{align}
  for all $z\in\C$ obeying $|\arg(z)|<\pi$ and $|\arg(1-z)|<\pi$, and the definition \cite[(15.2.2)]{NIST:DLMF} again. Combining \eqref{eq:defnuell3aux1} and \eqref{eq:hypergeomboundclose1} with
  $$
  (rs)^{-\zeta} \, \left(\frac{r\wedge s}{r\vee s}\right)^{\zeta}
  = \frac{1}{(r\vee s)^{2\zeta}}
  \sim \frac{1}{(r + s)^{2\zeta}}
  $$
  and the bound
  \begin{align}
    \begin{split}
      (r\vee s)^{-1-\alpha} |1-s^2/r^2|^{-1-\alpha}
      & = r^{1+\alpha}|r^2-s^2|^{-1-\alpha} \sim |r-s|^{-1-\alpha}
        \quad \text{for}\ 0 < s < r,
    \end{split}
  \end{align}
  we obtain \eqref{eq:defnuell3}.
  The monotonicity \eqref{eq:nuellmonotony} follows from \eqref{eq:defnuell3}.
\end{proof}

\subsection{Definition of quadratic forms on $\R_+$}
\label{s:defquadformsemigroup}

By general theory \cite[Section~1.3]{Fukushimaetal2011}, the quadratic form associated to $p_\zeta^{(\alpha)}(t,\cdot,\cdot)$,
\begin{align}
  \label{eq:defezeta}
  \begin{split}
    \ce_\zeta[u]
    & := \lim_{t\to0}\frac1t\langle u,u-p_\zeta^{(\alpha)}(t,\cdot,\cdot)u\rangle_{L^2(\R_+,r^{2\zeta}dr)},\quad u\in L^2(\R_+,r^{2\zeta}dr),
  \end{split}
\end{align}\index{$\ce_\zeta$}is well-defined. Moreover,
\begin{align}
  \label{eq:ezetalimit}
  \ce_\zeta[u] = \int_0^\infty dr\, r^{2\zeta} \int_0^\infty ds\, s^{2\zeta} |u(r)-u(s)|^2\,\nu_{\zeta}(r,s) \quad \text{for}\ \alpha<2,
\end{align}
by \eqref{eq:defnuell4} and dominated convergence, if the right-hand side is finite, or---in the opposite case---by Fatou's lemma.

\section{Ground state transformation for $\ce_\zeta$}
\label{s:optimalhardyalpha2}

The proof of our main result Theorem~\ref{gstransformreformulated} crucially relies on a ground state transform for $\ce_\zeta$, stated in Theorem~\ref{hbetagammatransformed} below, and the following computations.

\begin{proposition}
  \label{hbetagammatransformedcomputations}
  Let $\zeta\in(-1/2,\infty)$, $\alpha\in(0,2]$, $\gamma\in(-1,\infty)$, and $\beta\in(0,(2\zeta-\gamma)/\alpha)$.
  Then,
  \begin{align}
    \label{eq:defhbetagammaalphatransformed}
    \begin{split}
      h_{\beta,\gamma}(r)
      & := \int_0^\infty \frac{dt}{t}\, t^\beta \int_0^\infty ds\, s^\gamma p_\zeta^{(\alpha)}(t,r,s)
      = C^{(\alpha)}\left(\beta,\gamma,\zeta\right) r^{\alpha\beta+\gamma-2\zeta}, \quad r>0,
    \end{split}
  \end{align}
  where
  \begin{align}
    \label{eq:timeshiftconstant1resultalpha}
    C^{(\alpha)}(\beta,\gamma,\zeta)
    & := \frac{\Gamma(\beta)}{\Gamma(\frac{\alpha\beta}{2})} C\left(\frac{\alpha\beta}{2},\gamma,\zeta\right), \quad \text{with} \\
    \label{eq:timeshiftconstant1result}
    C(\beta,\gamma,\zeta)
    & := \frac{2^{-2\beta} \Gamma (\beta) \Gamma \left(\frac{1}{2} (\gamma+1)\right) \Gamma \left(\frac{1}{2} (2\zeta-2\beta-\gamma)\right)}{\Gamma \left(\frac{1}{2}(2\zeta-\gamma)\right) \Gamma\left(\frac{1}{2}(2\beta+\gamma+1)\right)}.
  \end{align}
  If we even have $\beta\in(1,(2\zeta-\gamma)/\alpha)$, then
  \begin{align}
    \label{eq:hardypottransformedalpha}
    \begin{split}
      q_{\beta,\gamma}(r)
      & := \frac{(\beta-1) h_{\beta-1,\gamma}(r)}{h_{\beta,\gamma}(r)} \\
      & \ = \frac{2^{\alpha} \Gamma \left(\frac{1}{2}(2\zeta+\alpha-\alpha\beta-\gamma)\right) \Gamma \left(\frac{1}{2}(\alpha\beta+\gamma+1)\right)}{\Gamma \left(\frac{1}{2}(2\zeta-\alpha\beta-\gamma)\right) \Gamma \left(\frac{1}{2}(\alpha\beta+\gamma+1-\alpha)\right)} r^{-\alpha}, \quad r>0.
    \end{split}
  \end{align}
\end{proposition}

\begin{theorem}
  \label{hbetagammatransformed}
  Let $\zeta\in(-1/2,\infty)$, $\alpha\in(0,2]\cap(0,2\zeta+1)$, and $\eta\in[0,\frac{2\zeta+1-\alpha}{2}]$.
  Furthermore, let 
  \begin{align}
    \label{eq:defpsietazeta}
    \Psi_\zeta(\eta) := \frac{2^\alpha\Gamma\left(\frac{2\zeta+1-\eta}{2}\right)\Gamma\left(\frac{\alpha+\eta}{2}\right)}{\Gamma\left(\frac{\eta}{2}\right)\Gamma\left(\frac{2\zeta+1-\eta-\alpha}{2}\right)}.
  \end{align}
  For $\gamma\in(-1,\infty)$ and $\beta\in(1,(2\zeta-\gamma)/\alpha)$, let $h_{\beta,\gamma}(r)$ and $q_{\beta,\gamma}(r)$ be as in \eqref{eq:defhbetagammaalphatransformed} and \eqref{eq:hardypottransformedalpha}, respectively.
  Then the following statements hold.
  \\
  {\rm(1)}
  If
  \begin{align}
    \label{eq:choicesigmabeta2resultalphatransformed}
    \beta
    & = \frac1\alpha\left(2\zeta-\gamma-\eta\right),
  \end{align}
  then, for all $r>0$,
  \begin{align}
    \label{eq:hardypottransformedalpha2}
    h_{\beta,\gamma}(r)
    = C^{(\alpha)}\left(\beta,\gamma,\zeta\right) r^{-\eta}
    \quad \text{and} \quad
    q_{\beta,\gamma}(r)
    = \Psi_\zeta(\eta)r^{-\alpha}.
  \end{align}
  \\
  {\rm(2)}
  Let $h(r):=r^{-\eta}$. Then, for every $u\in L^2(\R_+,r^{2\zeta}\,dr)$, we have the ground state representation
  \begin{align}
    \label{eq:hardyremainderagaintransformed}
    \begin{split}
      & \ce_\zeta[h\cdot u]
      = \Psi_\zeta(\eta)\int_0^\infty dr\, r^{2\zeta}\frac{|h(r)\cdot u(r)|^2}{r^\alpha} \\
      & \quad + \frac12\lim_{t\to0}\int_0^\infty dr\, r^{2\zeta}\int_0^\infty ds\, s^{2\zeta} \frac{p_\zeta^{(\alpha)}(t,r,s)}{t} |u(r) - u(s)|^2 h(r) h(s).
    \end{split}
  \end{align}
  If $\alpha<2$, then
  \begin{align}
    \label{eq:hardyremainderagaintransformedlimit}
    \begin{split}
      & \ce_\zeta[h\cdot u]
      = \Psi_\zeta(\eta)\int_0^\infty dr\, r^{2\zeta}\frac{|h(r)\cdot u(r)|^2}{r^\alpha} \\
      & \quad + \frac12\int_0^\infty dr\, r^{2\zeta}\int_0^\infty ds\, s^{2\zeta}\, \nu_{\zeta}(r,s) |u(r) - u(s)|^2 h(r) h(s).
    \end{split}
  \end{align}
\end{theorem}

Before we prove Proposition~\ref{hbetagammatransformedcomputations} and Theorem~\ref{hbetagammatransformed}, some remarks are in order.
\begin{remarks}
  (1) The computations of \eqref{eq:defhbetagammaalphatransformed} and \eqref{eq:hardypottransformedalpha} require $\gamma>-1$. Then, of course, $\mu(dr)=r^\gamma\,dr$ is a $\sigma$-finite measure.
  \\
  (2) Let $\zeta=(d-1+2\ell)/2$ and $\eta=\ell+\sigma$ with $d\in\N$ and $\ell>-1/2$. Since $\beta \in (1,\frac{2\zeta-\gamma}{\alpha})$ and $\gamma>-1$, the choice \eqref{eq:choicesigmabeta2resultalphatransformed} implies the restriction $\sigma\in(-\ell,d+\ell-\alpha)$. This should be compared with the admissible range $\sigma\in(-\ell-\alpha,(d-\alpha)/2]$ (or $\sigma\in[(d-\alpha)/2,d+\ell-\alpha)$, by symmetry) for $\alpha<2$ and $\sigma\leq(d-2)/2$ for $\alpha=2$, see Proposition \ref{propertiesphi}. Therefore, we require $\sigma\in(-\ell,(d-\alpha)/2]$, i.e., an attractive Hardy potential.
  \\
  (3) As in \cite[p.~237]{Bogdanetal2016}, one can show that the choice \eqref{eq:choicesigmabeta2resultalphatransformed} maximizes the prefactor (which depends on $\beta$ and $\gamma$ only through $\alpha\beta+\gamma$) on the right-hand side of \eqref{eq:hardypottransformedalpha}.
  \\
  (4) In the context of Theorem~\ref{gstransformreformulated}, we have $\zeta=(d-1+2\ell)/2$ and $\eta=\ell+\sigma$, and thus $\Psi_\zeta(\eta)=\Phi_\ell(\sigma)$. In particular, in view of \eqref{eq:sigmagammaellneg}, $\eta\mapsto\Psi_\zeta(\eta)$ is increasing for $\eta\leq(2\zeta+1-\alpha)/2$.
  \\
  (5) For $\zeta=(d-1+2\ell)/2$ and $\eta=\ell+\sigma$ with $d=1$ and $\ell=0$, the choice \eqref{eq:choicesigmabeta2resultalphatransformed} coincides precisely with the choice made in \cite[Section~1.3]{JakubowskiMaciocha2022} upon replacing $\gamma$ with $-\gamma$, and $\sigma$ with $\beta_{\rm JM}-\alpha/2$, where $\beta_{\rm JM}$ stands for the variable $\beta$ used in \cite{JakubowskiMaciocha2022}.
\end{remarks}

To prepare the proof of Proposition~\ref{hbetagammatransformedcomputations} when $\alpha\in(0,2)$, we note that for $\beta>0$ and $\tau>0$, the subordinator density function $\sigma_t^{(\alpha/2)}(\tau)$ satisfies
\begin{align}
  \label{eq:subordinatorpowerintegral}
  & \int_0^\infty t^\beta \sigma_t^{(\alpha/2)}(\tau)\,\frac{dt}{t}
    = \frac{\Gamma(\beta)}{\Gamma(\alpha\beta/2)\tau^{1-\alpha\beta/2}}, \quad \tau>0.
\end{align}
A calculation leading to \eqref{eq:subordinatorpowerintegral} can be found in \cite[Proof of Proposition 5]{Bogdanetal2016}. Note that the assumption $\beta<d/\alpha$ therein is irrelevant to \eqref{eq:subordinatorpowerintegral},
see also \cite[p.~1003]{JakubowskiWang2020}.

\begin{proof}[Proof of Proposition~\ref{hbetagammatransformedcomputations}]
  We begin with the proof of \eqref{eq:defhbetagammaalphatransformed}.
  We first consider $\alpha=2$ and note $C^{(2)}(\beta,\gamma,\zeta)=C(\beta,\gamma,\zeta)$. By scaling $s\mapsto sr$ and $t\mapsto r^2t$, we obtain
  \begin{align}
    \label{eq:defhbetagamma}
    \begin{split}
      h_{\beta,\gamma}(r)
      & = r^{2\beta+\gamma-2\zeta}\int_0^\infty \frac{dt}{t}\, t^\beta\int_0^\infty ds\, s^\gamma p_\zeta^{(2)}(t,1,s)
      = C(\beta,\gamma,\zeta) r^{2\beta+\gamma-2\zeta},
    \end{split}
  \end{align}
  with
  \begin{align}
    \label{eq:timeshiftconstant1}
    \begin{split}
      C(\beta,\gamma,\zeta)
      & = \int_0^\infty \frac{dt}{t}\, t^\beta\int_0^\infty ds\, s^{\gamma-\zeta} \cdot \frac{s^{1/2}}{2t}\me{-\frac{s^2+1}{4t}} I_{\zeta-1/2}\left(\frac{s}{2t}\right).
    \end{split}
  \end{align}
  To compute $C(\beta,\gamma,\zeta)$, we first integrate over $s$ (see Appendix~\ref{s:timeshiftconstantaux1}) and obtain
  \begin{align}
    \label{eq:timeshiftconstantaux1}
    \begin{split}
      & \int_0^\infty ds\, s^{\gamma-\zeta} \cdot \frac{s^{1/2}}{2t}\me{-\frac{s^2+1}{4t}} I_{\zeta-1/2}\left(\frac{s}{2t}\right) \\
      & \quad = 2^{\gamma-2\zeta} \Gamma \left(\frac{\gamma+1}{2}\right) t^{\frac{\gamma-2\zeta}{2}} \, _1\tilde{F}_1\left(\frac{2\zeta-\gamma}{2};\zeta+\frac12;-\frac{1}{4t}\right),
    \end{split}
  \end{align}
  whenever $\gamma>-1$. Plugging this into \eqref{eq:timeshiftconstant1} gives \eqref{eq:timeshiftconstant1result}, that is,
  \begin{align}
    \label{eq:timeshiftconstantaux2}
    \begin{split}
      & \int_0^\infty \frac{dt}{t}\, t^\beta \cdot 2^{\gamma-2\zeta} \Gamma \left(\frac{\gamma+1}{2}\right) t^{\frac{\gamma-2\zeta}{2}} \, _1\tilde{F}_1\left(\frac{2\zeta-\gamma}{2};\zeta+\frac12;-\frac{1}{4t}\right)
        = C(\beta,\gamma,\zeta),
    \end{split}
  \end{align}
  whenever $\beta\in(0,(2\zeta-\gamma)/2)$ (see Appendix~\ref{s:timeshiftconstantaux2}). This concludes the proof for $\alpha=2$. For $\alpha<2$, we use subordination and obtain
   \begin{align}
    \begin{split}
      h_{\beta,\gamma}(r)
      & = \int_0^\infty \frac{dt}{t}\, t^\beta \int_0^\infty ds\, s^\gamma p_\zeta^{(\alpha)}(t,r,s) \\
      & = \int_0^\infty ds\, s^\gamma\int_0^\infty d\tau\, p_\zeta^{(2)}(\tau,r,s) \int_0^\infty \frac{dt}{t}\, t^\beta \frac{\sigma_t^{(\alpha/2)}(\tau)}{d\tau} \\
      & = \frac{\Gamma(\beta)}{\Gamma(\frac{\alpha\beta}{2})} \int_0^\infty ds\, s^\gamma \int_0^\infty \frac{d\tau}{\tau}\, p_\zeta^{(2)}(\tau,r,s)\tau^{\frac{\alpha\beta}{2}} \\
      & = \frac{\Gamma(\beta)}{\Gamma(\frac{\alpha\beta}{2})} \cdot C\left(\frac{\alpha\beta}{2},\gamma,\zeta\right) r^{\alpha\beta+\gamma-2\zeta}.
    \end{split}
  \end{align}
  Here we first used \eqref{eq:subordination}, then \eqref{eq:subordinatorpowerintegral}, and then \eqref{eq:defhbetagamma}. This puts the constraints $\gamma>-1$ and $\beta\in(0,\frac{2\zeta-\gamma}{\alpha})$. The proof of \eqref{eq:defhbetagammaalphatransformed} is complete.
  
  \smallskip
  Formula~\eqref{eq:hardypottransformedalpha} follows from the computation in Item (1).
\end{proof}

\begin{proof}[Proof of Theorem~\ref{hbetagammatransformed}]
  \emph{Item (1).} This is obvious.

  \smallskip
  \emph{Item (2).} By Item~(1), Formula~\eqref{eq:hardyremainderagaintransformed} follows from \cite[Theorem~2]{Bogdanetal2016} with $X=\R_+$, $m(dr)=r^{2\zeta}\,dr$, $\mu(dr)=r^\gamma\,dr$ with $\gamma\in(-1,\frac{2\zeta-1-\alpha}{2})$, and $f(t)=t_+^{\beta-1}$ with $\beta=(2\zeta-\gamma-\eta)/\alpha>1$ as in \eqref{eq:choicesigmabeta2resultalphatransformed}.
  Formula \eqref{eq:hardyremainderagaintransformedlimit} follows from taking the limit $t\searrow0$ in \eqref{eq:hardyremainderagaintransformed} by using dominated convergence along with the definition $\lim_{t\searrow0}t^{-1} p_\zeta^{(\alpha)}(t,r,s)=\nu_{\zeta}(r,s)$ and the fact $t^{-1}p_\zeta^{(\alpha)}(t,r,s)\lesssim\nu_{\zeta}(r,s)$ (see Proposition~\ref{potentialexpansionsphericalharmonics}) if the right-hand side of \eqref{eq:hardyremainderagaintransformed} is finite, or---in the opposite case---Fatou's lemma.
\end{proof}

\section{Connection between the forms $\ce_\zeta$ and $\ce^{(\alpha)}$}
\label{s:notation}

In this section, we establish the link between $\ce_\zeta[u]$ and $\ce^{(\alpha)}[[u]_{\ell,m}]$, in order to prove Theorem~\ref{gstransformreformulated} via Theorem~\ref{hbetagammatransformed}.
To that end, we use, for $d\in\N$ and $\ell\in L_d$, the Fourier--Bessel transform $\F_\ell$,
\begin{align}
  \label{eq:deffourierbessel}
  v\mapsto (\F_\ell v)(k) := i^{-\ell}\int_0^\infty dr\, \sqrt{kr}J_{\frac{d-2}{2}+\ell}(kr) v(r),
  \quad k\in\R_+,
\end{align}\index{$\F_\ell$}initially defined for functions $v\in L^1(\R_+,dr)$. Note that the transform depends both on $d$ and $\ell$, but we omit the dimension $d$ from the notation since we regard $d$ as fixed. Note also that this transform is well-defined for all $d$ and $\ell$, for which $(d-1)/2+\ell>-1/2$.
The same arguments used to extend the Fourier transform on $L^1(\R^d)$ to $L^2(\R^d)$ allow one to extend the Fourier--Bessel transform to a unitary operator on $L^2(\R_+,dr)$. By abuse of notation, we denote this extension by $\F_\ell$, too, and observe the symmetry $\F_\ell=\F_\ell^*$, which is apparent from the definition \eqref{eq:deffourierbessel}.
The unitary operator $U:L^2(\R_+,r^{d-1+2\ell}dr)\to L^2(\R_+,dr)$ defined in \eqref{eq:defdoob}, the plane wave expansion \eqref{eq:spherewave2}, and orthonormality of $\{Y_{\ell,m}\}_{\ell,m}$ yield the intertwining relation
\begin{align}
  \label{eq:fourierbesselintertwine2}
  \left([u]_{\ell,m}\right)^\wedge(\xi) = \left(U^*[\F_\ell Uu]_{\ell,m}\right)(\xi),
  \quad \xi\in\R^d,\, u\in L^2(\R_+,r^{d-1+2\ell}dr).
\end{align}

The following lemma  says that the Fourier--Bessel transform diagonalizes $F(|D|)$, when restricted to a fixed $V_{\ell}$ or $V_{\ell,m}$. In the following, we write $F(|D|)^{1/2}$ as short for $|F(|D|)|^{1/2}\sgn(F(|D|))$.

\begin{lemma}
  \label{fourierbessel}
  Let $d\in\N$, $\ell,\ell'\in L_d$, $m\in M_\ell$, and $m'\in M_{\ell'}$. Let $F\in L_\loc^1(\R_+)$ and $Q(F(|D|)):=\{g\in L^2(\R^d):\, \int_{\R^d}|F(|\xi|)||\hat g(\xi)|^2\,d\xi<\infty\}$. Then, for $[u]_{\ell,m},\,[v]_{\ell',m'}\in Q(F(|D|))$,
  \begin{align}
    \label{eq:fourierbessel1}
    \begin{split}
      & \langle |F(|D|)|^{1/2} [v]_{\ell',m'},F(|D|)^{1/2} [u]_{\ell,m}\rangle_{L^2(\R^d)} \\
      & \quad = \delta_{\ell,\ell'}\delta_{m,m'}\langle \F_\ell Uv,F \cdot \F_\ell Uu\rangle_{L^2(\R_+,dr)}
    \end{split}
  \end{align}
  with the unitary operator $U:L^2(\R_+,r^{d-1+2\ell}dr)\to L^2(\R_+,dr)$ defined in \eqref{eq:defdoob}.
\end{lemma}

\begin{proof}
  By the orthonormality of $\{Y_{\ell,m}\}_{\ell\in L_d,m\in M_\ell}$, it suffices to consider $\ell=\ell'$ and $m=m'$.
  Then the claim follows from Plancherel's theorem, i.e.,
  \begin{align}
    \begin{split}
      & \langle |F(|D|)|^{1/2} [v]_{\ell',m'},F(|D|)^{1/2} [u]_{\ell,m}\rangle_{L^2(\R^d)}
        = \langle ([v]_{\ell',m'})^\wedge,F\cdot ([u]_{\ell,m})^\wedge \rangle_{L^2(\R^d)},
    \end{split}
  \end{align}
  and the intertwining relation \eqref{eq:fourierbesselintertwine2}.
\end{proof}

For simplicity in what follows, we occassionally identify operators in $L^2$ with their integral kernels. For instance, by a Schur test and \eqref{eq:normalizedalpha}, the maps
$L^2(\R_+,r^{2\zeta}dr)\ni f \mapsto \int_{\R_+}p_\zeta^{(\alpha)}(t,\cdot,s)f(s)\,s^{2\zeta}ds$ and
$L^2(\R^d)\ni f \mapsto \int_{\R^d}P^{(\alpha)}(t,\cdot,y)f(y)\,dy$
define symmetric contractions on $L^2(\R_+,r^{2\zeta}dr)$ and $L^2(\R^d)$, respectively.
In this connection, see also Corollary~\ref{kalfcor} below.

The following proposition says that the kernel $p_\zeta^{(\alpha)}(t,r,s)$ arises through the Fourier--Bessel transform of $\me{-tk^\alpha}$ when $\ell=\zeta-(d-1)/2$.

\begin{proposition}
  \label{representationpzetafourierbessel}
  Let $\alpha\in(0,2]$ and $\zeta\in(-1/2,\infty)$. Then for all $r,s,t>0$, one has
  \begin{align}
    \label{eq:representationpzetafourierbessel}
    \begin{split}
      p_\zeta^{(\alpha)}(t,r,s)
      & = (rs)^{-\zeta} \int_0^\infty dk\, \me{-tk^\alpha}\sqrt{kr}\sqrt{ks}J_{\zeta-\frac12}(kr)J_{\zeta-\frac12}(ks).
    \end{split}
  \end{align}
\end{proposition}

\begin{proof}
  For $\alpha=2$, this integral is tabulated, e.g., in \cite[Formula (2.12.39.3)]{PrudnikovetalVol21988}.
  For $\alpha<2$, the claim follows from subordination \eqref{eq:subordination}.
\end{proof}

We are now in a position to link the quadratic form $\ce_\zeta$ on $L^2(\R_+,r^{2\zeta}dr)$ to the form $\ce^{(\alpha)}$ on $V_{\ell,m}$. Moreover, we show that $\nu_{\zeta}$, defined in \eqref{eq:defnuell1}, equals the radial part of the projection of $\nu$ onto $V_\ell$ with $\ell=\zeta-(d-1)/2\in L_d$; see \eqref{eq:nuellmainthm}.

\begin{proposition}
  \label{relationformheatkernel}
  Let $d\in\N$, $\alpha\in(0,2]$, $\ell\in L_d$, $m\in M_\ell$, and $\zeta=(d-1)/2+\ell$.
  Then the following statements hold. \\
  \textup{(1)} For all $u\in L^2(\R_+,r^{d-1+2\ell}dr)$, we have
  \begin{align}
    \label{eq:linksemigroupsrdrplus}
    \begin{split}
      \langle[u]_{\ell,m},P^{(\alpha)}(t,\cdot,\cdot)[u]_{\ell,m}\rangle_{L^2(\R^d)}
      & = \langle u, p_\zeta^{(\alpha)}(t,\cdot,\cdot)u\rangle_{L^2(\R_+,r^{2\zeta}dr)},
        \quad t>0
    \end{split}
  \end{align}
  and
  \begin{align}
    \label{eq:relationformheatkernel3}
    \ce^{(\alpha)}[[u]_{\ell,m}] = \ce_\zeta[u].
  \end{align}
  \textup{(2)} For all $r,s,t>0$, we have
  \begin{align}
    \label{eq:linksemigroupsrdrpluspointwise}
     p_\zeta^{(\alpha)}(t,r,s)
    = (rs)^{-\ell} \iint\limits_{\bs^{d-1}\times\bs^{d-1}} \overline{Y_{\ell,m}(\omega_x)} Y_{\ell,m}(\omega_y) P^{(\alpha)}(t,r\omega_x,s\omega_y)\,d\omega_x\,d\omega_y.
  \end{align}
  \textup{(3)} Let further $\alpha\in(0,2)$ and $\nu_{\zeta}(r,s)$ be defined in \eqref{eq:defnuell1}. Then for all $r,s>0$ with $r\neq s$, we have
  \begin{align}
    \label{eq:defnuell}
    \begin{split}
      \nu_{\frac{d-1}{2}+\ell}(r,s)
      & = (rs)^{-\ell} \iint\limits_{\bs^{d-1}\times\bs^{d-1}}\overline{Y_{\ell,m}(\omega_x)}Y_{\ell,m}(\omega_y)\nu(r\omega_x,s\omega_y)\,d\omega_x\,d\omega_y.
    \end{split}  
  \end{align}
\end{proposition}

\begin{proof}
  (1) Formula~\eqref{eq:linksemigroupsrdrplus} follows from Lemma~\ref{fourierbessel} and Proposition~\ref{representationpzetafourierbessel}.
  Recalling the definitions \eqref{eq:fractlaplaceheatkernel} and \eqref{eq:defezeta} of $\ce^{(\alpha)}[[u]_{\ell,m}]$ and $\ce_\zeta[u]$ through the semigroups $P^{(\alpha)}$ and $p_\zeta^{(\alpha)}$, Formula~\eqref{eq:relationformheatkernel3} follows from \eqref{eq:linksemigroupsrdrplus}.
  \\
  (2) Formula \eqref{eq:linksemigroupsrdrpluspointwise} follows from the Fourier representation of $P^{(\alpha)}(t,x,y)$, the plane wave expansion for $\me{ix\cdot\xi}$ (Corollary~\ref{kalfcor}), and the observation \eqref{eq:representationpzetafourierbessel}.
  \\
  (3) Formula~\eqref{eq:defnuell} follows from the pointwise limit $\lim_{t\searrow0}t^{-1}P^{(\alpha)}(t,x,y) = \nu(x,y)$, \eqref{eq:linksemigroupsrdrpluspointwise}, and the dominated convergence theorem to interchange the limit $t\to0$ with the integral over $\bs^{d-1}\times\bs^{d-1}$. The application of the dominated convergence theorem is justified because $t^{-1}P^{(\alpha)}(t,x,y)\lesssim\nu(x,y)<\infty$ (see \cite{Bogdanetal2014}), and because the integration in \eqref{eq:defnuell} is over a bounded set.
\end{proof}

\section{Proof of Theorem~\ref{gstransformreformulated}}
\label{s:proofsgstransformreformulated}

For $\alpha\in(0,2)$, Theorem~\ref{gstransformreformulated} follows from \eqref{eq:relationformheatkernel3} in Proposition~\ref{relationformheatkernel} and \eqref{eq:hardyremainderagaintransformedlimit} in Theorem~\ref{hbetagammatransformed} with $\zeta=(d-1+2\ell)/2$ and $\eta=\ell+\sigma$. Formula~\eqref{eq:nuellmainthm} is the content of Proposition~\ref{relationformheatkernel}(3).

\smallskip
It remains to consider $\alpha=2$. With $h(r)=r^{-\ell-\sigma}$, we have
\begin{align}
  \begin{split}
    |(\nabla [h \cdot u]_{\ell,m})(x)|^2
    & = \left||x|^{-\sigma}u'(|x|)-\sigma|x|^{-\sigma-1}u(|x|)\right|^2|Y_{\ell,m}(\omega_x)|^2 \\
    & \quad + |x|^{-2\sigma-2}|u(|x|)|^2\, |x|^2|\nabla Y_{\ell,m}(\omega_x)|^2.
  \end{split}
\end{align}
Since the eigenvalues of $-\Delta_{\bs^{d-1}}$ are $\ell(d+\ell-2)$, we have
\begin{align}
  \int_{\bs^{d-1}}|x|^2|\nabla Y_{\ell,m}(\omega_x)|^2\,d\omega_x
  = \ell(d+\ell-2).
\end{align}
Thus, recalling $\Phi_\ell(\sigma)=\ell(d+\ell-2)+\sigma(d-\sigma-2)$, it suffices to show
\begin{align}
  \begin{split}
    & \int_0^\infty dr\, r^{d-1}\left[\sigma^2 r^{-2\sigma-2}|u(r)|^2 - 2\sigma \re\left(r^{-2\sigma-1}\overline{u(r)}u'(r)\right)\right] \\
    & \quad = \sigma(d-\sigma-2)\int_0^\infty dr\, r^{d-1-2\sigma}\, \frac{|u(r)|^2}{r^2}.
  \end{split}
\end{align}
Indeed, from an integration by parts, we get
\begin{align}
  \begin{split}
    & -\int_0^\infty dr\, r^{d-1-2\sigma-1}\overline{u(r)}u'(r) \\
    & \quad = \int_0^\infty dr\, r^{d-2-2\sigma}\overline{u'(r)}u(r) + (d-2-2\sigma)\int_0^\infty dr\, r^{d-3-2\sigma}|u(r)|^2
  \end{split}
\end{align}
and therefore
\begin{align}
  \begin{split}
    & \int_0^\infty dr\, r^{d-1}\left[\sigma^2 r^{-2\sigma-2}|u(r)|^2 - 2\sigma \re\left(r^{-2\sigma-1}\overline{u(r)}u'(r)\right)\right] \\
    & \quad = \int_0^\infty dr\, r^{d-3-2\sigma}|u(r)|^2 \cdot\left[\sigma^2 + \sigma(d-2-2\sigma)\right] \\
    & \quad = \sigma(d-\sigma-2)\int_0^\infty dr\, r^{d-1-2\sigma}\, \frac{|u(r)|^2}{r^2},
  \end{split}
\end{align}
as claimed. This concludes the proof of Theorem \ref{gstransformreformulated} for $\alpha=2$. \qed

\appendix

\section{Plane wave expansion}
\label{s:fourierbessel}

We begin with an expansion result for sufficiently smooth functions on $\bs^{d-1}$. See Kalf \cite{Kalf1995} for an extensive review about this and related results.

\begin{proposition}
  \label{kalfreview}
  Let $d,n\in\N$ and $f\in C^{n}(\bs^{d-1})$. Then the Laplace series representation
  \begin{align}
    \label{eq:kalfreview}
    f(\omega) = \sum_{\ell\in L_d}\sum_{m\in M_\ell}\langle Y_{\ell,m},f\rangle_{L^2(\bs^{d-1})}\, Y_{\ell,m}(\omega)
  \end{align}
  {\rm (1)} converges uniformly with respect to $\omega\in\bs^{d-1}$, if $n=\floor{(d-1)/2}\vee1$; and \\
  {\rm (2)} converges absolutely, uniformly with respect to $\omega\in\bs^{d-1}$, if $n=2\floor{(d+3)/4}$.
\end{proposition}

In the following, let $d\in\N$, $\ell\in L_d$, and $m\in M_\ell$.
For $x,\xi\in\R^d\setminus\{0\}$, we have
\begin{align}
  \frac{1}{(2\pi)^{d/2}}\int_{\bs^{d-1}}d\omega_x \overline{Y_{\ell,m}(\omega_x)}\, \me{ix\cdot\xi}
  = i^\ell\frac{J_{\frac{d-2}{2}+\ell}(|x||\xi|)}{(|x||\xi|)^{(d-2)/2}} \cdot\overline{Y_{\ell,m}(\omega_\xi)}.
\end{align}
This is contained, e.g., in Lapidus \cite{Lapidus1982} for $d=2$, in Messiah \cite[(B.105)]{Messiah1969} for $d=3$, and in Avery and Avery \cite[(2.17), (3.1), (3.25), (3.50), (3.68), (4.24)-(4.25)]{AveryAvery2018} for $d\geq4$. Thus, by Proposition~\ref{kalfreview}, for any fixed $x\in\R^d\setminus\{0\}$, we have the Laplace series representation of $\me{i\xi\cdot x}$, often called \emph{plane wave expansion},
\begin{align}
  \label{eq:spherewave2}
  \begin{split}
    \frac{\me{i a \xi\cdot x}}{(2\pi)^{d/2}}
    & = \sum_{\ell\in L_d}i^\ell \frac{J_{\frac{d-2}{2}+\ell}(|a||x||\xi|)}{(|a||x||\xi|)^{(d-2)/2}}\sum_{m\in M_\ell} \overline{Y_{\ell,m}(\sgn(a)\omega_\xi)}Y_{\ell,m}(\omega_x) \\
    & \quad \text{for all}\ a\in\R\setminus\{0\}, \quad x,\xi\in\R^d\setminus\{0\},
  \end{split}
\end{align}
which is absolutely and uniformly convergent with respect to $\omega_\xi\in\bs^{d-1}$.

Formula \eqref{eq:spherewave2} has the following useful consequence for spherically symmetric Fourier symbols $\ck(\xi)=\ck(|\xi|)$. For simplicity, below we suppose $\ck\in L^1(\R^d)$, but the idea readily generalizes to more general $\ck$ under suitable conditions. 

\begin{corollary}
  \label{kalfcor}
  Let $d\in\N$ and $\ck\in L^1(\R^d)$ be spherically symmetric. Let $K(-i\nabla)=\F^{-1}\ck\F$ be the Fourier multiplier corresponding to $\ck$. Then the integral kernel $K(x-y)$ of $K(-i\nabla)$ has the uniformly and absolutely convergent expansion
  \begin{align}
    \label{eq:kalfcor1}
    K(x-y)
    = \sum_{\ell\in L_d}\sum_{m\in M_\ell} K_{\ell}(|x|,|y|)Y_{\ell,m}(\omega_x)\,\overline{Y_{\ell,m}(\omega_y)}
  \end{align}
  for all $x,y\in\R^d\setminus\{0\}$ with $x\neq y$, with coefficients
  \begin{align}
    \label{eq:kalfcor0}
      K_{\ell}(r,s)
      & := (rs)^{-(d-2)/2}\int_0^\infty k\, \ck(k) J_{\frac{d-2}{2}+\ell}(kr)J_{\frac{d-2}{2}+\ell}(ks)\,dk \\
      \label{eq:kalfcor2}
      & \ = \int_{\bs^{d-1}\times\bs^{d-1}}d\omega_x\,d\omega_y\, \overline{Y_{\ell,m}(\omega_x)}\, Y_{\ell,m}(\omega_y) K(r\omega_x-s\omega_y),
    \quad m\in M_\ell.
  \end{align}
\end{corollary}

\begin{proof}
  Formula \eqref{eq:kalfcor1} follows from Proposition~\ref{kalfreview} since
  \begin{align}
    \begin{split}
      & K(x-y)
      = \frac{1}{(2\pi)^d}\int_{\R^d}\me{i\xi\cdot(x-y)} \ck(\xi)\,d\xi \\
      & \quad = \sum_{\ell,\ell';m,m'} Y_{\ell,m}(\omega_x)\,\overline{Y_{\ell',m'}(\omega_y)}\, \int_0^\infty k^{d-1}\, \ck(k) \frac{J_{\frac{d-2}{2}+\ell}(k|x|)}{(k|x|)^{(d-2)/2}} \frac{J_{\frac{d-2}{2}+\ell'}(k|y|)}{(k|y|)^{(d-2)/2}}\,dk \\
      & \qquad \times i^{\ell-\ell'} \int_{\bs^{d-1}}d\omega_\xi\, \overline{Y_{\ell,m}(\omega_\xi)}\,Y_{\ell',m'}(\omega_\xi) \\
      & \quad = \sum_{\ell,m}Y_{\ell,m}(\omega_x)\,\overline{Y_{\ell,m}(\omega_y)} (|x||y|)^{-(d-2)/2}\int_0^\infty k\, \ck(k) J_{\frac{d-2}{2}+\ell}(kr)J_{\frac{d-2}{2}+\ell}(ks)\,dk.
    \end{split}
  \end{align}
  Formula~\eqref{eq:kalfcor2} follows from this and the orthonormality of $\{Y_{\ell,m}\}_{\ell,m}$.
\end{proof}

\section{Properties of the $\alpha/2$-stable subordination density}
\label{s:propertiessubordinator}
For $\alpha\in(0,2)$ and $t>0$ the function $[0,\infty)\ni\lambda\mapsto \me{-\lambda^{\alpha/2}}$ is completely monotone and, therefore, is the Laplace transform
\begin{align}
  \label{eq:laplacetransform}
  \me{-t\lambda^{\alpha/2}} = \int_0^\infty \me{-\tau\lambda}\, \sigma_t^{(\alpha/2)}(\tau)\,d\tau
\end{align}
of a 
non-negative density $\sigma_t^{(\alpha/2)}(\tau)\,d\tau$, which is called the $\frac{\alpha}{2}$-stable subordination transition density. We now collect some properties of $\sigma_t^{(\alpha/2)}(\tau)$ needed for our analysis. See, e.g., \cite{Pollard1946,Bergstrom1952,IbragimovLinnik1971,Schillingetal2012} for references.

\smallskip
\noindent
(1) The measure $\sigma_t^{(\alpha/2)}(\tau)\,d\tau$ is normalized, that is,
\begin{align}
  \label{eq:subordinatornormalized}
  \int_0^\infty \sigma_t^{(\alpha/2)}(\tau)\,d\tau = 1, \quad t>0.
\end{align}
\\
\smallskip
(2) By \cite[(4)--(5)]{Pollard1946} (see also \cite[(2)]{Bergstrom1952}), we have the convergent expansions
\begin{align}
  \label{eq:subordinatorexplicit}
  \begin{split}
    \sigma_t^{(\frac\alpha2)}(\tau)
    & = \frac{t^{-\frac2\alpha}}{\pi} \int_0^\infty \me{-\frac{\tau}{t^{2/\alpha}}u}\exp\left(-u^{\alpha/2}\cos\left(\frac{\alpha\pi}{2}\right)\right)\cdot \sin\left(u^{\alpha/2} \sin\left(\frac{\alpha\pi}{2}\right)\right)\,du \\
    & = \frac{1}{\pi}\sum_{k=1}^\infty (-1)^{k+1}\frac{\Gamma(k \alpha/2+1)}{k!}\cdot\frac{t^k}{\tau^{1+k\alpha/2}}\cdot\sin\left(\frac{\pi k \alpha}{2}\right) \geq0, \quad t,\tau>0.
  \end{split}
\end{align}
In particular,
\begin{align}
  \label{eq:subordinatorquotientexplicit}
  \lim_{t\searrow0}\frac{\sigma_t^{(\frac\alpha2)}(\tau)}{t} = \frac{\Gamma(\alpha/2+1)\,\sin(\pi\alpha/2)}{\pi\,\tau^{1+\alpha/2}}, \quad \tau>0.
\end{align}
The second line in \eqref{eq:subordinatorexplicit} was first derived (formally) by Humbert \cite{Humbert1945} and later proved by Pollard \cite[(4)]{Pollard1946}, see also Bergstr\"om \cite{Bergstrom1952} and \cite{Skorohod1954,Skorohod1961,Gawronski1988,MontrollBendler1984,MontrollWest1979} for further classic treatments, as well as \cite[(5.18)-(5.19)]{Bogdanetal2009} and \cite[(A.15)-(A.19)]{PaulBaschnagel2013} (and the references therein) for textbook references.

\smallskip
\noindent
(3) From \eqref{eq:subordinatorexplicit}, we read off the scaling
\begin{align}
  \label{eq:subordinatorscaling}
  t^{2/\alpha} \sigma_t^{(\alpha/2)}(t^{2/\alpha}\tau) = \sigma_1^{(\alpha/2)}(\tau).
\end{align}
\smallskip
(4) The density function $\sigma_t^{(\alpha/2)}(\tau)$ is a convolution semigroup \cite[p.~48--49]{Schillingetal2012}, i.e.,
\begin{align}
  \label{eq:subordinatorconvolution}
  \begin{split}
    \int_0^\infty d\tau \int_0^\infty d\tau'\, \sigma_t^{(\alpha/2)}(\tau)\, \sigma_{t'}^{(\alpha/2)}(\tau')\, f(\tau+\tau')
    & = \int_0^\infty d\tau\, (\sigma_t^{(\alpha/2)}\ast\sigma_{t'}^{(\alpha/2)})(\tau)f(\tau) \\
    & = \int_0^\infty d\tau\, \sigma_{t+t'}^{(\alpha/2)}(\tau)f(\tau), \quad f\in L^1.
  \end{split}
\end{align}
\\
\smallskip
(5) Finally, we recall sharp estimates for $\sigma_t^{(\alpha/2)}(\tau)$. From \eqref{eq:subordinatorexplicit},
\begin{align}
  \label{eq:subordinatorupperbounds}
  \sigma_{1}^{(\alpha/2)}(\tau)
  \sim \tau^{-1-\alpha/2}, \quad \tau\gtrsim1.
\end{align}
On the other hand,
\begin{align}
  \label{eq:subordinatorsmalltimes}
  \sigma_{1}^{(\alpha/2)}(\tau)
  \sim \sqrt{\frac{(\alpha/2)^{1/(1-\alpha/2)}}{2\pi(1-\alpha/2)}} \cdot \frac{\exp\left(-(1-\frac\alpha2)\left(\frac{\alpha}{2\tau}\right)^{\alpha/(2-\alpha)}\right)}{\tau^{(2-\alpha/2)/(2-\alpha)}}, \quad \tau\lesssim1.
\end{align}
This appeared first in Ibragimov and Linnik \cite[Theorem~2.4.6]{IbragimovLinnik1971} (see also their formula (2.4.1) which is the second line in \eqref{eq:subordinatorexplicit}) and Hawkes \cite[Lemma 1]{Hawkes1971}, and was recently rederived by Ercolani, Jansen, and Ueltschi \cite[(4.6)-(4.8), Theorem 4.4]{Ercolanietal2019} and Grzywny, Lezaj, and Trojan \cite[Theorems A and 4.17]{Grzywnyetal2023} in greater generality. In particular, $\sigma_{1}^{(\alpha/2)}(\tau)$ vanishes at $\tau=0$ faster than any power of $\tau$, which can also be derived from the first line in \eqref{eq:subordinatorexplicit} and a non-stationary phase computation. We summarize Formulae \eqref{eq:subordinatorupperbounds}--\eqref{eq:subordinatorsmalltimes} as follows.

\begin{proposition}
  \label{subordinatorbounds}
  Let $\alpha\in(0,2)$ and $D(\alpha):=(1-\frac\alpha2)\left(\frac{\alpha}{2}\right)^{\alpha/(2-\alpha)}$. Then,
  \begin{align}
    \label{eq:subordinatorbounds}
    \begin{split}
      \sigma_{1}^{(\alpha/2)}(\tau)
      & \sim_{\alpha} \frac{\exp\left(-D(\alpha)\cdot\tau^{-\alpha/(2-\alpha)}\right)}{\tau^{(2-\alpha/2)/(2-\alpha)}}\one_{\tau<1} + \tau^{-1-\alpha/2}\one_{\tau>1}, \quad \tau>0
    \end{split}
  \end{align}
  and
  \begin{align}
    \label{eq:subordinatorbounds2}
    \begin{split}
      \frac{\sigma_t^{(\alpha/2)}(\tau)}{t}
      & \sim_\alpha \frac{t^{\frac{\alpha-1}{2-\alpha}}\,\exp\left(-D(\alpha)\left(\frac{t^{2/\alpha}}{\tau}\right)^{\alpha/(2-\alpha)}\right)}{\tau^{(2-\alpha/2)/(2-\alpha)}}\one_{\tau<t^{2/\alpha}} + \tau^{-1-\frac\alpha2}\one_{\tau>t^{2/\alpha}} \\
      & \lesssim_\alpha \tau^{-1-\frac\alpha2}, \quad t,\tau>0.
    \end{split}
  \end{align}
\end{proposition}

\section{Auxiliary computations}
\label{s:omitted}

\subsection{Proof of \eqref{eq:normalizedalpha} for $\alpha=2$}
\label{s:normalized}
Writing $\beta:=r/(2t)^{1/2}$, we obtain
\begin{align}
  \begin{split}
    & \int_0^\infty ds\, s^{2\zeta} p_\zeta^{(2)}(t,r,s)
    = (2t)^{-\frac12}r^{-\zeta}\int_0^\infty ds\, s^{\zeta} \cdot \sqrt{\frac{rs}{2t}} \me{-\frac{r^2+s^2}{4t}} \cdot I_{\zeta-1/2}\left(\frac{rs}{2t}\right) \\
    & \quad = \beta^{-2\zeta-1}\int_0^\infty ds\, s^{\zeta+1/2} \exp\left(-\frac{\beta^2+s^2/\beta^2}{2}\right) \cdot I_{\zeta-1/2}(s)
    = 1
  \end{split}
\end{align}
by scaling $s\mapsto s\cdot 2t/r$. The last equality follows from \cite[(10.43.23)]{NIST:DLMF}, i.e.,
\begin{align}
  \int_0^\infty t^{\nu+1} \exp(-p^2t^2) \cdot I_{\nu}(bt)\,dt
  = \frac{b^\nu}{(2p^2)^{\nu+1}} \exp\left(\frac{b^2}{4p^2}\right)
\end{align}
for $\re(\nu)>-1$ and $\re(p^2),b>0$. Here $b=1$, $\nu=\zeta-1/2$, and $p^2=1/(2\beta^2)$.

\subsection{Proof of \eqref{eq:chapman} for $\alpha=2$}
\label{s:chapman}

We recall \cite[(10.43.28)]{NIST:DLMF}, i.e.,
\begin{align*}
  \int_0^\infty z\,\me{-p^2z^2}I_\rho(az)I_\rho(bz)\,dz
  = \frac{1}{2p^2}\exp\left(\frac{a^2+b^2}{4p^2}\right)I_\rho\left(\frac{ab}{2p^2}\right),
  \quad \rho>-1, \quad \re(p^2)>0.
\end{align*}
Using this formula with $\rho=\zeta-1/2$, $p=\sqrt{\frac{t+t'}{4tt'}}$, $a=\frac{r}{2t}$, and $b=\frac{s}{2t'}$, we obtain
\begin{align*}
  & \int_0^\infty p_\zeta^{(2)}(t,r,z)p_{\zeta}^{(2)}(t',z,s)z^{2\zeta}\,dz \\
  & \ = \frac{(rs)^{1/2-\zeta}}{4tt'}\exp\left(-\frac{r^2}{4t}-\frac{s^2}{4t'}\right) \int_0^\infty z\, \exp\left(-z^2\cdot\frac{t+t'}{4tt'}\right)I_{\zeta-1/2}\left(\frac{rz}{2t}\right) I_{\zeta-1/2}\left(\frac{sz}{2t'}\right)\,dz \\
  & \ = \frac{(rs)^{1/2-\zeta}}{2(t+t')}\exp\left(-\frac{r^2+s^2}{4(t+t')}\right)\cdot I_{\zeta-1/2}\left(\frac{rs}{2(t+t')}\right)
    = p_\zeta^{(2)}(t+t',r,s),
\end{align*}
as desired.

\subsection{Proof of \eqref{eq:defnuell2}}
\label{s:selfdomconvsubordinatordensity}

For $r,s>0$ with $r\neq s$, we will verify that
\begin{align}
  \begin{split}
    & I_\zeta^{(\alpha)}(r,s)
      := \int_0^\infty \frac{\Gamma(\tfrac\alpha2+1)\sin(\tfrac{\pi\alpha}{2})}{\pi\,\tau^{1+\alpha/2}}\cdot \frac{(rs)^{\frac12-\zeta}}{2\tau}\exp\left(-\frac{r^2+s^2}{4\tau}\right)I_{\zeta-\frac12}\left(\frac{rs}{2\tau}\right)\,d\tau
  \end{split}
\end{align}
equals the right-hand side of \eqref{eq:defnuell2}. By reflecting $\tau\mapsto1/\tau$, it suffices to compute
\begin{align}
  \int_0^\infty \frac{d\tau}{\tau}\, \tau^{1+\alpha/2} \cdot \exp\left(-\tau\, \frac{r^2+s^2}{4}\right)I_{\zeta-\frac12}\left(\tau\cdot \frac{rs}{2}\right)\,d\tau.
\end{align}
To that end, we use \cite[(2.15.3.2)]{PrudnikovetalVol21988}, i.e.,
\begin{align}
  \begin{split}
    & \int_0^\infty \frac{d\tau}{\tau}\, \tau^{\tilde\alpha}\,\me{-p\tau}\,I_\nu(c\tau) \\
    & \quad = p^{-\tilde\alpha-\nu}\, \left(\frac{c}{2}\right)^{\nu}\, \frac{\Gamma(\nu+\tilde\alpha)}{\Gamma(\nu+1)} \, _2F_1\left(\frac{\tilde\alpha+\nu}{2}, \frac{\tilde\alpha+\nu+1}{2}; \nu+1; \frac{c^2}{p^2}\right)
  \end{split}
\end{align}
with $\re(\tilde\alpha+\nu)>0$, $|\re(c)|<\re(p)$.
In our situation, $p=(r^2+s^2)/4$, $c=rs/2$, $\tilde\alpha=1+\alpha/2$, and $\nu=\zeta-1/2>-1$. Thus, the conditions of \cite[(2.15.3.2)]{PrudnikovetalVol21988} are satisfied and we obtain
\begin{align}
  \begin{split}
    I_\zeta^{(\alpha)}(r,s)
    & = 2^{1+\alpha} \frac{\Gamma\left(\frac\alpha2+1\right)\,\sin\left(\frac{\pi\alpha}{2}\right)}{\pi} \, \frac{\Gamma(\zeta+\frac\alpha2+\frac12)}{\Gamma(\zeta+\frac12)}\, (r^2+s^2)^{-(\zeta+\frac\alpha2+\frac12)} \\
    & \quad \times\, _2F_1\left(\frac{\zeta+\alpha/2+1/2}{2}, \frac{\zeta+\alpha/2+3/2}{2}; \zeta+\frac12; \frac{4(rs)^2}{(r^2+s^2)^2}\right),
  \end{split}
\end{align}
as desired.

\subsection{Proof of \eqref{eq:timeshiftconstantaux1}}
\label{s:timeshiftconstantaux1}

To prove \eqref{eq:timeshiftconstantaux1}, we scale $s\mapsto 2ts$ and obtain
\begin{align}
  \label{eq:timeshiftconstantaux1aux1}
  \begin{split}
    & \int_0^\infty ds\, s^{\gamma-\zeta} \cdot \frac{s^{\frac12}}{2t} \me{-\frac{s^2+1}{4t}} I_{\zeta-\frac12}\left(\frac{s}{2t}\right)
    = \me{-\frac{1}{4t}}(2t)^{\frac12+\gamma-\zeta}\int_0^\infty ds\, \me{-ts^2} s^{\gamma-\zeta+\frac12}I_{\zeta-\frac12}(s).
  \end{split}
\end{align}
Now we use \cite[(10.25.2)]{NIST:DLMF}, i.e.,
\begin{align}
  \label{eq:besseliseries}
  I_{\zeta-1/2}(\tau)
  = \left(\frac{\tau}{2}\right)^{\zeta-1/2} \sum_{k=0}^\infty \frac{(\tau^2/4)^k}{k!\Gamma(\zeta+1/2+k)}, \quad \tau\in\C\setminus(-\infty,0],
\end{align}
and obtain, for $\gamma>-1$,
\begin{align}
  \int_0^\infty ds\, \me{-ts^2} s^{\gamma-\zeta+\frac12}I_{\zeta-1/2}(s)
  = 2^{-(\zeta+1/2)}t^{-(\gamma+1)/2}\sum_{k=0}^\infty \frac{2^{-2k}\Gamma\left(\frac{\gamma+1+2k}{2}\right)t^{-k}}{k! \Gamma\left(\zeta+\frac12+k\right)}.
\end{align}
Therefore, the left-hand side of \eqref{eq:timeshiftconstantaux1aux1} is
\begin{align}
  \begin{split}
    & \me{-\frac{1}{4t}}(2t)^{\frac12+\gamma-\zeta} \cdot 2^{-(\zeta+\frac12)}t^{-(\gamma+1)/2}\sum_{k=0}^\infty \frac{2^{-2k}\Gamma\left(\frac{\gamma+1+2k}{2}\right)t^{-k}}{k! \Gamma\left(\zeta+\frac12+k\right)} \\
    & \quad = 2^{\gamma-2\zeta}t^{\frac12(\gamma-2\zeta)} \cdot \me{-\frac{1}{4t}}\sum_{k=0}^\infty \frac{\Gamma\left(\frac{\gamma+1+2k}{2}\right)}{k! \Gamma\left(\zeta+\frac12+k\right)}\left(\frac{1}{4t}\right)^k.
  \end{split}
\end{align}
Thus, it suffices to show
\begin{align}
  \label{eq:timeshiftconstantaux1aux2}
  \begin{split}
    & \me{-\frac{1}{4t}}\sum_{k=0}^\infty \frac{\Gamma\left(\frac{\gamma+1+2k}{2}\right)}{k! \Gamma\left(\zeta+\frac12+k\right)}\left(\frac{1}{4t}\right)^k
      = \Gamma \left(\frac{\gamma+1}{2}\right) \cdot \, _1{\tilde F}_1\left(\frac{2\zeta-\gamma}{2};\zeta+\frac12;-\frac{1}{4t}\right).
  \end{split}
\end{align}
Using \cite[(13.2.4), (13.2.39)]{NIST:DLMF}, i.e., $_1\tilde F_1(a;b;z)=\,_1\tilde F_1(b-a;b;-z)\me{z}$ for all $z\in\C$, we see that the right-hand side equals
\begin{align}
  \label{eq:timeshiftconstantaux1aux3}
  \Gamma \left(\frac{1}{2} (\gamma+1)\right) \cdot \, _1{\tilde F}_1\left(\frac{1}{2}(\gamma+1);\zeta+\frac12;\frac{1}{4t}\right)\me{-1/(4t)}.
\end{align}
Combining this with \eqref{eq:timeshiftconstantaux1aux2} shows that it suffices to prove
\begin{align}
  \label{eq:timeshiftconstantaux1aux4}
  \begin{split}
    & \sum_{k=0}^\infty \frac{\Gamma\left(\frac{\gamma+1+2k}{2}\right)}{k! \Gamma\left(\zeta+\frac12+k\right)}\left(\frac{1}{4t}\right)^k
      = \Gamma \left(\frac{1}{2}(\gamma+1)\right) \cdot \, _1{\tilde F}_1\left(\frac{1}{2}(\gamma+1);\zeta+\frac12;\frac{1}{4t}\right).
  \end{split}
\end{align}
To this end, we recall \cite[(13.2.2)]{NIST:DLMF}
\begin{align}
  _1\tilde F_1(a;b;z)
  = \frac{1}{\Gamma(a)}\sum_{k=0}^\infty\frac{\Gamma(a+k)}{k!\Gamma(b+k)}z^k,
  \quad z\in\C.
\end{align}
Thus, the right-hand side of \eqref{eq:timeshiftconstantaux1aux4} is
\begin{align}
  \sum_{k=0}^\infty\frac{\Gamma(\frac{\gamma+1+2k}{2})}{k!\Gamma(\zeta+\frac12+k)}\left(\frac{1}{4t}\right)^k,
\end{align}
which indeed equals the left-hand side of \eqref{eq:timeshiftconstantaux1aux4}.

\subsection{Proof of \eqref{eq:timeshiftconstantaux2}}
\label{s:timeshiftconstantaux2}

To prove \eqref{eq:timeshiftconstantaux2}, it suffices to show the last line in
\begin{align}
  \label{eq:timeshiftconstantaux2aux1}
  \scriptsize
  \begin{split}
    & \int_0^\infty \frac{dt}{t}\, t^{\beta} \cdot 2^{\gamma-2\zeta} t^{\frac12(\gamma-2\zeta)}\cdot \Gamma \left(\frac{1}{2}(\gamma+1)\right) \cdot \, _1\tilde{F}_1\left(\frac{1}{2}(2\zeta-\gamma); \zeta+\frac12;-\frac{1}{4t}\right) \\
    & \quad = 2^{\gamma-2\zeta} \Gamma\left(\frac{1}{2}(\gamma+1)\right) \int_0^\infty \frac{dt}{t} t^{\beta-\zeta+\frac{\gamma}{2}} \cdot \, _1\tilde{F}_1\left(\frac{1}{2}(2\zeta-\gamma);\zeta+\frac12;-\frac{1}{4t}\right) \\
    & \quad = 2^{-2\beta} \cdot \Gamma \left(\frac{1}{2}(\gamma+1)\right) \int_0^\infty \frac{dt}{t} t^{\zeta-\beta-\frac{\gamma}{2}} \cdot \, _1\tilde{F}_1\left(\frac{1}{2}(2\zeta-\gamma);\zeta+\frac12;-t\right) \\
    & \quad = \frac{2^{-2(\beta+1)} \Gamma (\beta ) \Gamma \left(\frac{1}{2}(\gamma+1)\right) \Gamma \left(\frac{1}{2}(2\zeta-2\beta-\gamma)\right)}{\Gamma \left(\frac{1}{2}(2\zeta-\gamma)\right) \Gamma \left(\frac{1}{2}(2\beta+\gamma+1)\right)}.
  \end{split}
\end{align}
Here we substituted $t\mapsto t^{-1}$ and then $t\mapsto4t$ to deduce the third from the second line. To prove the last line, we use \cite[(13.10.10)]{NIST:DLMF}, i.e.,
\begin{align}
  \label{eq:intmonomialconfluenthypergeometric}
  \int_0^\infty \frac{dt}{t}t^\delta\ _1\tilde{F}_1(a;b;-t)\,dt
  = \frac{\Gamma(a-\delta)\Gamma(\delta)}{\Gamma(a)\Gamma(b-\delta)},
  \quad 0<\re(\delta)<\re(a),
\end{align}
and set
$\delta=\frac12(2\zeta-2\beta-\gamma)$,
$a=\frac{1}{2}(2\zeta-\gamma)$, and
$b=\zeta+\frac12$.
Thus, our parameters must satisfy
$0<2\zeta-2\beta-\gamma<\zeta-\gamma$, that is,
$\beta\in(0,(2\zeta-\gamma)/2)$.
This ends the proof.

\printindex


\def\cprime{$'$}

\end{document}